\documentclass[11pt]{article}

\usepackage{mathrsfs}

\usepackage{amsmath,amsthm,amssymb}
\usepackage{graphicx}
\usepackage{authblk}
\usepackage{appendix}
\usepackage[font={small,rm}]{caption}

\font\tencmmib=cmmib10 \skewchar\tencmmib '60
\newfam\cmmibfam
\textfont\cmmibfam=\tencmmib

\def\lessim{\ \lower4pt\hbox{$
		\buildrel{\displaystyle <}\over\sim$}\ }
\def\gessim{\ \lower4pt\hbox{$\buildrel{\displaystyle >}
		\over\sim$}\ }

\def\la{\langle}
\def\ra{\rangle}

\newcommand{\e}{\mathbb{E}}
\newcommand{\p}{\mathbb{P}}

\newtheorem{lemma}{\bf Lemma}
\newtheorem{definition}{\bf Definition}
\newtheorem{theorem}{\bf Theorem}

\newtheorem{remark}{\bf Remark}

\newtheorem{proposition}{\bf Proposition}

\newenvironment{Proof of lemma}{\noindent{\bf Proof of Lemma}}{\hfill$\Box$\newline}
\newenvironment{Proof of theorem}{\noindent{\bf Proof of Theorem}}{\hfill{\footnotesize${\square}$}\newline}
\newenvironment{Proof of theorems}{\noindent{\bf Proof of Theorems}}{\hfill$\Box$\newline}
\newenvironment{Proof of proposition}{\noindent{\bf Proof of Proposition}}{\hfill$\Box$\newline}
\newenvironment{Proof of propositions}{\noindent{\bf Proof of Propositions}}{\hfill$\Box$\newline}
\newenvironment{Proof of exercise}{\noindent{\it Proof of Exercise:}}{\hfill$\Box$}
\begin{document}
	
		\title{Phase transition in the spiked random tensor with Rademacher prior}
		
		\author{
			Wei-Kuo Chen\thanks{School of Mathematics, University of Minnesota. Email: wkchen@umn.edu}
		}

		\maketitle

		\begin{abstract}		
		We consider the problem of detecting a deformation from a symmetric Gaussian random $p$-tensor $(p\geq 3)$ with a rank-one spike sampled from the Rademacher prior. Recently in Lesieur et al. \cite{LMLKZ+17}, it was proved that there exists a critical threshold $\beta_p$ so that when the signal-to-noise ratio exceeds $\beta_p$, one can distinguish the spiked and unspiked tensors and weakly recover the prior via the minimal mean-square-error method. On the other side, Perry, Wein, and Bandeira \cite{PWB17} proved that there exists a $\beta_p'<\beta_p$ such that any statistical hypothesis test can not distinguish these two tensors, in the sense that their total variation distance asymptotically vanishes, when the signa-to-noise ratio is less than $\beta_p'$. In this work, we show that $\beta_p$ is indeed the critical threshold that strictly separates the distinguishability and indistinguishability between the two tensors under the total variation distance. Our approach is based on a subtle analysis of the high temperature behavior of the pure $p$-spin model with Ising spin, arising initially from the field of spin glasses. In particular, we identify the signal-to-noise criticality $\beta_p$ as the critical temperature, distinguishing the high and low temperature behavior, of the Ising pure $p$-spin mean-field spin glass model.
		% and it is also the critical value for weak recovery of the signal. 
		
		%and to control the total variation distance through a sharp upper bound for the fluctuation of the free energy.
		\end{abstract}
	
		\tableofcontents
		
		\section{Introduction}\label{sec1}

		The problem of detecting a deformation in a symmetric Gaussian random tensor is concerned about whether there exists a statistical hypothesis test that can reliably distinguish the deformation from the noise. In the matrix case, if $W$ is a Gaussian Wigner ensemble and $u$ is a unit vector, the goal is to distinguish the unspiked matrix $W$ and the spiked matrix $W+\beta uu^T$ for a given signal-to-noise ratio $\beta.$ It is well-known that in this case the top eigenvalue of the spiked matrix exhibits the so-called BBP (Baik, Ben Arous, and P\'ech\'e) transition \cite{BBP,BGN11,CDMF09,PS06}. Namely, the top eigenvalue successfully detects the signal if the strength of $\beta$ exceeds the critical threshold $1$, while it fails to provide indicative information if $\beta<1.$ It was further proved in \cite{MRZ17,OMH13,PWBM16} that for any $\beta<1,$ every statistical hypothesis test can not distinguish the spiked and unspiked matrices. 
		
		In recent years, the above phenomena were also studied in the spiked symmetric Gaussian random $p$-tensor model of the form, $W+\beta u^{\otimes p}$. Earlier results were obtained by Montanari and Richard \cite{MR+14} and Montanari, Reichman, and Zeitouni \cite{MRZ17} in the setting of spherical prior, where they showed that there exist $\beta_p^-$ and $\beta_p^+$ with $\beta_p^-<\beta_p^+$ such that if the signal-to-noise ratio exceeds $\beta_p^+$, it is possible to distinguish the spiked and unspiked tensors and weakly recover the signal, but these are impossible if the signal-to-noise ratio is less than $\beta_p^-.$ More recently, Lesieur et al. \cite{LMLKZ+17} studied this detection problem for any prior by means of the minimal mean-square-error (MMSE). They discovered that there exists a critical threshold $\beta_p$ (depending on the prior) such that when the signal-to-noise ratio exceeds the critical value $\beta_p$, one can distinguish the two tensors by the MMSE and weakly recover the signal by the MMSE estimator. In contrast, when $\beta<\beta_p$, MMSE fails to distinguish the two tensors and weakly recovery of the signal is not possible. On the other side, Perry, Wein, and Bandeira \cite{PWB17} studied the detection problem with spherical, Rademacher, and sparse Rademacher priors and they provided an improvement on the bounds in \cite{MRZ17,MR+14}. Moreover, their results showed that in each of these three cases, there exists a $\beta_p'<\beta_p$ such that for any $\beta<\beta_p'$, it is impossible to distinguish the two tensors in the sense that the total variation distance between the spiked and unspiked tensors asymptotically vanishes. As a consequence, every statistical hypothesis test fails to distinguish the two tensors. 
		%Additionally, under the same regime $\beta<\beta_p',$ it was also proved that weak recovery of the signal is not possible as every output from the spiked tensor is uncorrelated with the signal in the limit. 
		The paper \cite{PWB17} then left with a conjecture that indistinguishability between the two tensors should be valid up to the critical threshold $\beta_p$.
		%, suggested by the MMSE method in \cite{LMLKZ+17}.
		%In this regime, they also showed that weak recovery of the signal is impossible. Additionally, they obtained a weak form of recovery for the signal in this regime, where it was proved that the %spiked tensor outputs a vector with positive correlation with the signal. In sharp contrast, if $\beta<\beta_p^-$, the two tensors are indistinguishable that the total variation distance %between these two tensors is asymptotically zero.  Their paper then left with a replica prediction, which states that these two different phenomena, for $\beta<\beta_p^-$ and $\beta>\beta_p^+$, %hold up to a critical threshold that is determined through the {\it replica symmetric ansatz} in the pure $p$-spin mean-field spin glass model.
		
		The aim of this paper is to study the symmetric Gaussian random $p$-tensor ($p\geq 3$) with Rademacher prior as in the setting of \cite{PWB17}. We show that the threshold $\beta_p$, suggested by the MMSE method in \cite{LMLKZ+17}, is indeed the critical value that strictly separates the distinguishability and indistinguishability between the spiked and unspiked tensors under the total variation distance. More precisely, it is established that when the signal-to-noise ratio is less than the critical value $\beta_p,$ the total variation distance between the spiked and unspiked tensors converges to zero. This establishes the aforementioned prediction in \cite{PWB17}. In particular, we identify the critical value $\beta_p$ as the critical temperature, distinguishing the high and low temperature behavior, of the Ising pure $p$-spin mean-field spin glass model. 
		
		%In addition, whenever the signal-to-noise ratio is smaller than $\beta_p$, we also show that weak recovery of the signal is not possible. Namely, for any (random) probability measure produced %by the spiked tensor, its sampling is uncorrelated with the signal in the limit.
		
		Our approach is based on the methodologies originated from the study of mean-field spin glasses, especially those for the Sherrington-Kirkpatrick model and the mixed $p$-spin models, see \cite{Pan,Talbook1,Talbook2}. Roughly speaking, spin glass models are disordered spin systems initially invented by theoretical physicists  in order to explain the strange magnetic behavior of certain alloys, such as CuMn. Mathematically, they are usually formulated as stochastic processes with high complexity and present several crucial features, e.g., quenched disorder and frustration, that are commonly shared in many real world problems, involving combinatorial random optimizations.  Over the past decades, the study of spin glasses has received a lot of attention in both physics and mathematics communities, see \cite{MPV} for physics overview and \cite{Pan,Talbook1,Talbook2} for mathematical development. 
		
		One way to investigate the detection problem in the symmetric Gaussian random tensor is through the total variation distance between the spiked and unspiked tensors. While in the detection problem $\beta$ represents the signal-to-noise ratio, we regard $\beta$ as a (inverse) temperature parameter in the pure $p$-spin model. Notably, under this setting, the ratio of the densities between the two tensors can be computed as the partition function of the pure $p$-spin model with temperature $\beta.$ In \cite{MRZ17,PWB17}, the authors controlled the total variation distance by the second moment of the partition function. Different than their consideration, we relate this distance to the free energy of the pure $p$-spin model with Ising spin, see Lemma \ref{thm7:lem3}. This relation allows us to show that the critical threshold $\beta_p $ can be determined by the critical temperature of the pure $p$-spin model. In bounding the total variation distance, the most critical ingredient is played by a sharp upper bound concerning the fluctuation of the free energy up to the critical temperature for all $p\geq 3$. In the case $p=2$, the pure $p$-spin model is famously known as the Sherrington-Kirkpatrick model and its free energy was shown to possess a Gaussian central limit theorem in the weak limit up to the critical temperature $\beta_p =1$ by Aizenman, Lebowitz, and Ruelle \cite{ALR}. As for even $p\geq 4$, Bovier, Kurkova, and L\"owe \cite{BKL} showed that the same result also holds (with different scaling than that in the Sherrington-Kirkpatrick model), but not up to the critical temperature. Our main contribution is that we obtain a sharp upper bound for the fluctuation of the free energy, which is comparable to the one in  \cite{BKL} and more importantly it is valid up to the critical temperature for all $p\geq 3$ including odd $p.$ This allows us to extract a sharp upper bound for the total variation distance and deduce the desired result. 
		
		Besides the consideration of the detection problem, we also present some new results and arguments for the pure $p$-spin models that are of independent interest in spin glasses. First, we show that if the temperature is below the critical value $\beta_p $, the model presents the high temperature or {\it replica symmetric} solution in the sense that any two independently sampled spin configurations from the Gibbs measure are essentially orthogonal to each other by  providing exponential tail probability and moment controls. While these results can also be established at {\it very} high temperature by some well-known techniques in spin glasses, such as the cavity method, the second moment method, and Latala's argument (see \cite[Chapter 1]{Talbook1}), it is relatively a more challenging task to obtain the same behavior throughout the entire high temperature regime. We show that this is achievable in the pure $p$-spin model (see Theorem \ref{extra:thm2} and \ref{extra:thm3}) and indeed, our method can also be applied to more general situations, the mixed $p$-spin models (see Remark \ref{rmk1}). Next, in terms of technicality, our argument for the above result is based on the Guerra-Talagrand replica symmetry breaking bound for the coupled free energy for two systems. This bound has been playing a critical role in the study of the mixed {\it even} $p$-spin models, see Talagrand \cite{Talbook2}. Its validity for the model involving odd $p$ mixture is however generally unknown as it is unclear whether the error term along Guerra's replica symmetry breaking interpolation possesses a nonnegative sign or not. To tackle this obstacle, we adopt the synchronization property, introduced by Panchenko \cite{PanMultipSP,Panchenko2015}, that the overlap matrix is asymptotically symmetric and positive semi-definite under the Gibbs average, which was established heavily relying on the fact that the Ghirlanda-Guerra identities imply ultrametricity of the overlaps \cite{PanUltra}. This allows us to show that the error term creates a nonnegative sign and ultimately leads to the validity of the Guerra-Talagrand bound in the pure odd $p$-spin model if one restricts the functional order parameters to be of {\it one-step replica symmetry breaking}. Whether this bound is valid for more general functional order parameters remains open.
		
		For other related works on the detection problem of spiked matrices and tensors, we invite the readers to check a variety of low rank matrix estimation problems, including explicit characterizations of mutual information in \cite{BDM+16,KM+09,KXZ+16,LM+16,M+17} and the performance of the approximate message passing (AMP) for MMSE method  and compressed sensing in \cite{BDM+16,BM+11,DMM+09,DMM+10,JM+13}.  More recently, the fluctuation of the likelihood ratio in the spiked Wigner model was studied in \cite{AKJ+17}. In the case of spiked tensors, phase transitions of the mutual information and the MMSE estimator were recently studied for any $p$ and given prior, see \cite{BM+17,LMLKZ+17} for symmetric case and \cite{BMM+17} for non-symmetric case. The performance of the AMP in the spiked tensor was also investigated in \cite{LMLKZ+17}. See also \cite{BDMK+16,BKMMZ+17,BMDK+17,DAM+16,DM+14,MR+16,RP+16}. For the study of Gaussian random $p$-tensor in terms of complexity, see \cite{BMMN+17}.

		This paper is organized as follows. In Section \ref{sec2}, we state our main results on the detection problem. In Section \ref{sec3}, their proofs are presented  and are essentially self-contained for those who wish to learn only the roles of spin glass results in the detection problem. The high temperature results on the pure $p$-spin model are gathered and established in Section \ref{sec4} with great details.

	   \section{Main results}\label{sec2}
	  
	  We begin by setting some standard notation. Let $p\in \mathbb{N}.$ For any $N\geq 1,$ denote by  $\Omega_N:={\otimes}_{j=1}^p\mathbb{R}^N$ the space of all real-valued tensors $Y=(y_{i_1,\ldots,i_p})_{1\leq i_1,\ldots,i_p\leq N}$. For any two tensors $Y$ and $Y'$, their outer product and inner product are defined respectively by
	  \begin{align*}
	  (Y\otimes Y')_{i_1,\ldots,i_p,i_1',\ldots,i_p'}&=y_{i_1,\ldots,i_p}y_{i_1',\ldots,i_p'}
	  \end{align*}
	  and
	  $$\bigl\la Y,Y'\bigr\ra:=\sum_{i_1,\ldots,i_p=1}^Ny_{i_1,\ldots,i_p}y_{i_1,\ldots,i_p}'.$$ For $\tau\in \mathbb{R}^N,$ we define $\tau^{\otimes p}=\tau\otimes\cdots\otimes \tau\in  \Omega_N$ as the $p$-th order power of $\tau.$  For any permutation $\pi \in \mathfrak{S}_p$ and any tensor $Y$, $Y^\pi\in \Omega^N$ is defined as $y_{i_1,\ldots,i_p}^{\pi}=y_{i_{\pi(1)},\ldots,i_{\pi(p)}}$. We say that a tensor $Y$ is symmetric if $Y=Y^{\pi}$ for all $\pi \in \mathfrak{S}_p.$ Let $\Omega_N^s$ be the collection of all symmetric tensors. For any measurable $A\subseteq \mathbb{R}^k$ for some $k\geq 1$, we use $\mathscr{B}(A)$ to stand for the Borel $\sigma$-field on $A.$
	  
	  %For instance, $Y=u^{\otimes p}$ for some $u=(u_1,\ldots,u_p)$ is a symmetric tensor.
	  %where $\mathbb{S}_N$ is the $(N-1)$-dimensional unit sphere in $\mathbb{R}^N.$
	  
	  The symmetric Gaussian random tensor is defined as follows. Denote by $Y\in \Omega_N$ a random tensor with i.i.d. entries $y_{i_1,\ldots,i_p}\thicksim N(0,2/N).$ Define a symmetric random  tensor $W$ by
	  \begin{align*}
	  W=\frac{1}{p!}\sum_{\pi\in \mathfrak{S}_p}Y^\pi.
	  \end{align*}
	  For example, if $p=2$, $W$ is the Gaussian orthogonal ensemble, i.e., $w_{ij}=w_{ji}$ are independent of each other with $w_{ij}\thicksim N(0,1/N)$ for $i<j$ and $w_{ii}\thicksim N(0,2/N).$ Let $S_N:=\{\pm 1/\sqrt{N}\}^N$. Assume that $u$ is sampled uniformly at random from $S_N$ and is independent of $W.$ Set the spiked random tensor as
	  $$
	  T=W+\beta  u^{\otimes p}
	  $$
	  for $\beta \geq 0.$  Let $$d_{TV}(W,T):=\sup_{A\in \mathscr{B}(\Omega_N^s)}|\p(W\in A)-\p(T\in A)|$$ 
	  be the total variation distance between $W$ and $T$. We now make the distinguishability and indistinguishability between $W$ and $T$ precise.
	  \begin{definition}\label{def1} We say that	
	  	\begin{itemize}
	  		\item[$(i)$] $W$ and $T$ are indistinguishable if $\lim_{N\rightarrow\infty}d_{TV}(W,T)=0$.
	  		%\item $W$ and $T$ are weakly distinguishable if $\lim_{N\rightarrow\infty}d_{TV}(W,T)\in (0,1).$
	  		\item[$(ii)$]  $W$ and $T$ are distinguishable if $\lim_{N\rightarrow\infty}d_{TV}(W,T)=1.$
	  	\end{itemize}

	  \end{definition} 
  
     Let $u$ be a realization of the prior. For a given tensor $X\in \Omega_N^s$, consider the detection problem that under the null hypothesis, $X=W$ and under the alternative hypothesis, $X=T.$ Item $(i)$ essentially says that any statistical test can not reliably distinguish these two hypotheses. Item $(ii)$ means there exists a sequences of events that distinguishes these two tensors. Next we define the notion of weak recovery for $u$.
	 
	 \begin{definition}\label{def2}
	 For $\beta>0,$ we say that weak recovery of $u$ is possible if there exists a sequence of random probability measures $\mu_{N}$ on $\Omega_N^s\times \mathscr{B}(S_N)$ and a constant $c>0$ such that
	 \begin{align}\label{def2:eq1}
	 \lim_{N\rightarrow\infty}\p\Bigl(\int_{S_N}|\la u,\tau\ra| \mu_N(T,d\tau)\geq c\Bigr)=1
	 \end{align} 
	 and that weak recovery of $u$ is not possible if for any  random probability measure $\mu_{N}$ on $\Omega_N^s\times \mathscr{B}(S_N)$ and constant $c>0$,
	 	\begin{align}\label{def2:eq2}
	 	\lim_{N\rightarrow\infty}\p\Bigl(\int_{S_N}|\la u,\tau\ra| \mu_N(T,d\tau)\geq c\Bigr)=0.
	 	\end{align} 
	 Here $\mu_N$ is a random probability measure on $\Omega_N^s\times \mathscr{B}(S_N)$ means that $\mu_N$ is a mapping from $\Omega_N^s\times \mathscr{B}(S_N)$ to $[0,1]$ such that $\mu_N(\cdot,A)$ is  $\mathscr{B}(\Omega_N^s)$-measurable for each $A\in \mathscr{B}(S_N)$ and $\mu_N(w,\cdot)$ is a probability measure on $(S_N, \mathscr{B}(S_N))$ for each $w\in \Omega_N^s.$
	 \end{definition}
 
   A few comments are in position. Consider a given realization of signal $u$ and tensor $T$. Equation \eqref{def2:eq1} ensures that there exists some $\tau$ produced though the measure $\mu_{N}(T,d\tau)$ such that $u$ and $\tau$ have a nontrivial overlap. To understand \eqref{def2:eq2}, let $\phi:\Omega_N^s\rightarrow S_N$ be any measurable function. If we consider the random probability measure $\mu_N(w,\tau)$ defined by $$
   \mu_N(w,\tau)=\left\{
   \begin{array}{ll}
   1,&\mbox{if $\tau=\phi(w)$},\\
   0,&\mbox{if $\tau\neq \phi(w)$},
   \end{array}
   \right.
   $$
   then from \eqref{def2:eq2}, for any $c>0,$
   $$
   \lim_{N\rightarrow\infty}\p\bigl(|\la u,\phi(T)\ra|\geq c\bigr)=0.
   $$
   In other words, any vector generated by $T$ is uncorrelated with the signal $u$ and thus, it does not provide indicative information about $u.$ We emphasize that Definitions \ref{def1} and \ref{def2} are not directly related to each other. Nevertheless, we will show that both of them hold up to a critical threshold in our main result below.

	  Now we introduce the pure $p$-spin model. For each $N\geq 1,$ set $\Sigma_N:=\{\pm 1\}^N$. The Hamiltonian of the pure $p$-spin model is defined by
	  \begin{align*}
	  H_N(\sigma)&=\frac{1}{N^{(p-1)/2}}\sum_{1\leq i_1,\ldots,i_p\leq N}g_{i_1,\ldots,i_p}\sigma_{i_1}\cdots\sigma_{i_p}
	  \end{align*}
	  for $\sigma\in\Sigma_N$. Its covariance can be computed as
	  \begin{align*}
	  \e H_N(\sigma^1)H_N(\sigma^2)=N(R_{1,2})^p,
	  \end{align*}
	  where $R_{1,2}$ is the overlap between $\sigma^1,\sigma^2\in\Sigma_N$,
	  \begin{align}\label{op}
	  R_{1,2}&:=\frac{1}{N}\sum_{i=1}^N\sigma_i^1\sigma_i^2.
	  \end{align} 
	  In the terminology of detection problem, $\beta$ is understood as the signal-to-noise ratio. In the pure $p$-spin model, we regard $\beta$ as a (inverse) temperature parameter. For a given   temperature $\beta\geq 0,$ define the free energy by
	  \begin{align}\label{fg}
	  F_N(\beta)&=\frac{1}{N}\log \sum_{\sigma\in\Sigma_N}\frac{1}{2^N}\exp \frac{\beta}{\sqrt{2}} H_N(\sigma).
	  \end{align}
	  If $p=2$, this model is  known as the Sherrington-Kirkpatrick model and it has been intensively studied over the past decades. The readers are referred to check the books \cite{Pan,Talbook1,Talbook2} for recent mathematical advances for the SK model as well as more general models involving a mixture of pure $p$-spin interactions. In particular, it is already known that the thermodynamic limit of $F_N$ ($N\rightarrow\infty$) converges to a nonrandom quantity that can be expressed as the famous Parisi formula, see, e.g., \cite{Pan00,Tal03}. Denote this limit by $F$. A direct application of Jensen's inequality to \eqref{fg} implies that for all $\beta\geq 0,$
	  \begin{align*}
	  %\label{ineq}
	  F(\beta)\leq \frac{\beta^2}{4}.
	  \end{align*}
	  Define the high temperature regime for the pure $p$-spin model as
	  \begin{align*}
	  \mathcal{R}=\Bigl\{\beta> 0: F(\beta)=\frac{\beta^2}{4}\Bigr\}.
	  \end{align*}
	  Set the critical temperature by 
	  \begin{align*}
	  \beta_p&:=\max \mathcal{R}.
	  \end{align*}	 
	  Our main result shows that $\beta_p$ is the critical threshold in the detection problem.
	  
	  \begin{theorem}
	  	\label{thm7} Let $p\geq 3.$ The following statements hold:
	  	\begin{itemize}
	  		\item[$(i)$] If $0<\beta <\beta_p,$ then $W$ and $T$ are indistinguishable and weak recovery of $u$ is impossible.
	  		\item[$(ii)$] If $\beta >\beta_p$, then $W$ and $T$ are distinguishable and weak recovery of $u$ is possible.
	  	\end{itemize}
	  	  \end{theorem}

  	      Our main contribution in Theorem \ref{thm7} is the part on the indistinguishability of $W$ and $T$ in the statement $(i)$. Previous results along this line were established in \cite{PWB17}, where the authors showed that there exists some $\beta_p'<\beta_p$ so that $W$ and $T$ are indistinguishable for any $\beta<\beta_p'$. Theorem~\ref{thm7}$(i)$ here proves that this behavior is indeed valid up to the critical value $\beta_p$. As one will see from Theorem \ref{extra:thm1} below, we give a characterization of the high temperature regime $\mathcal{R}$ and provide one way to simulate $\beta_p$ in terms of an auxiliary function deduced from the optimality of the Parisi formula for the free energy at high temperature. Numerically, it is obtained that
  	      $$
  	      \begin{tabular}{l|c}
  	      $p$&$\beta_p$\\ \hline
  	      $3$&$1.535$\\ \hline
  	      $4$&$1.621$\\ \hline
  	      $5$&$1.647$\\ \hline
  	      $6$&$1.657$
  	      \end{tabular}$$ 
  	      This agrees with the prediction in \cite[Figure 1]{PWB17}. We comments that for Theorem \ref{thm7}$(i)$, a polynomial rate of convergence for the total variation distance can also be obtained, see Remark \ref{rmk0} below. In comparison, we add that Theorem \ref{thm7} is quite different from the BBP transition for $p=2$, see \cite{BBP,MRZ17,OMH13,PWBM16}. In this case, $\beta_2=1$ and it is known that for $\beta>\beta_2$, one can distinguish $W$ and $T$ in the sense of Definition \ref{def1} by using the top eigenvalue. For $0<\beta<\beta_2,$ it presents a weaker sense of distinguishability, $\lim_{N\rightarrow\infty}d_{TV}(W,T)\in (0,1)$, see Remark \ref{rmk-1}.
  	      
  	         As mentioned before, the work \cite{LMLKZ+17} investigated the present detection problem for any given prior. Their results state that one can distinguish $W$ and $T$ by the MMSE method and weakly recover the signal through the MMSE estimator when $\beta>\beta_p$; if $\beta<\beta_p,$ they concluded that weak recovery of the signal is not possible. In other words, their results imply the weak recovery part of item $(i)$ as well as the statement of item $(ii)$.
  	         %\footnote{In \cite{PWB17}, weak recovery of $u$ is defined in a slightly different sense than Definition \ref{def2}. They said that weak recovery of $u$ is impossible if for any measurable %$\phi:\Omega_N^s\rightarrow S_N$ and $c>$ such that $\p\bigl(\la u,\phi(T)\ra^p\geq c\bigr)\rightarrow 0.$} 
  	         Nevertheless, we emphasize that their approach and the way how the critical value $\beta_p$ was discovered are fundamentally different from the argument we present here. As one will see, while Theorem \ref{thm7}$(ii)$ follows directly from a relation (see Lemma \ref{thm7:lem3}) between the total variation distance and the free energies, the delicate part is Theorem \ref{thm7}$(i)$, which is the major component of this paper.

      \medskip
      
  {\noindent \bf Acknowledgements.} The author thanks G. Ben Arous for introducing this project to him and A. Auffinger and A. Jagannath for some fruitful discussions. He is indebted to D. Panchenko for asking good questions that lead to a major improvement on the result of weak recovery of $u$ and several suggestions regarding the presentation of the paper. He also thanks J. Barbier, A. Montanari, and L. Zdeborov\'a for suggesting several references and valuable comments. Research partly supported by NSF DMS-1642207 and Hong Kong Research Grants Council GRF-14302515.

        \section{Proof of Theorem \ref{thm7}}\label{sec3}
        
        \subsection{Total variation distance}
        
        In this subsection, we prepare some lemmas for Theorem \ref{thm7}. Recall that $Y$ is a random tensor with i.i.d. $y_{i_1,\ldots,i_p}\thicksim N(0,2/N)$ and $W$ is the symmetric random tensor generated by $Y.$ Throughout the remainder of the paper, we use $I(S)$ to standard for an indicator function on a set $S.$ We first establish an elementary expression of the total variation distance.
        
        \begin{lemma}\label{thm7:lem1}
        	Let $U,V$ be two $n$-dimensional random vectors with densities $f_U$ and $f_V$ satisfying $f_U(r)\neq 0$ and $f_V(r)\neq 0$ a.e. on $\mathbb{R}^n.$ Then
        	\begin{align*}
        	d_{TV}(U,V)
        	&=\int_0^1 \p\Bigl(\frac{f_U(V)}{f_V(V)}< x\Bigr)dx=\int_0^1\p\Bigl(\frac{f_U(U)}{f_V(U)}>\frac{1}{x}\Bigr)dx.
        	\end{align*}
        \end{lemma}
        
        \begin{proof}
        	Note that 
        	\begin{align*}
        	d_{TV}(U,V)&=\frac{1}{2}\int_{\mathbb{R}^n} |f_V(r)-f_U(r)|dr=\int_{f_U(r)\leq f_V(r)}\bigl(f_V(r)-f_U(r)\bigr)dr.
        	\end{align*}
        	Using Fubini's theorem and this equation, the first equality follows by
        	\begin{align*}
        	\int_0^1 \p\Bigl(\frac{f_U(V)}{f_V(V)}< x\Bigr)dx&=\int_0^1\int_{\mathbb{R}^n} I\Bigl(\frac{f_U(r)}{f_V(r)}<x\Bigr)f_V(r)drdx\\
        	&=\int_{\mathbb{R}^n} \int_0^1I\Bigl(\frac{f_U(r)}{f_V(r)}<x\Bigr)dxf_V(r)dr\\
        	&=\int_{\mathbb{R}^n} I\Bigl(\frac{f_U(r)}{f_V(r)}\leq 1\Bigr)\Bigl(1-\frac{f_U(r)}{f_V(r)}\Bigr)f_V(r)dr\\
        	&=\int_{f_U(r)\leq f_V(r)}\bigl(f_V(r)-f_U(r)\bigr)dr.
        	\end{align*}
            To obtain the second equality, one simply exchanges the roles of $U,V$.
        \end{proof}

        %\begin{lemma}
        %	\label{thm7:lem2}
        %	For any $\tau\in \mathbb{R}^N$, we have that
        %	\begin{align*}
        %	\bigl\la W,\tau^{\otimes p}\bigr\ra&=\bigl\la Y,\tau^{\otimes p}\bigr\ra.
        %	\end{align*}
        %\end{lemma}
        
        %\begin{proof}
        %	Observe that for any $\tau\in \mathbb{S}_N$,
        %	\begin{align*}
        %	\bigl\la W,\tau^{\otimes p}\bigr\ra
        %	&=\sum_{i_1,\ldots,i_p=1}^Nw_{i_1,\ldots,i_p}\tau_{i_1}\cdots\tau_{i_p}\\
        %	&=\sum_{i_1,\ldots,i_p=1}^N\Bigl(\frac{1}{p!}\sum_{\pi \in \mathfrak{S}_p}y_{i_{\pi(1)},\ldots,i_{\pi(p)}}\Bigr)\tau_{i_1}\cdots\tau_{i_p}\\
        %	&=\frac{1}{p!}\sum_{\pi \in \mathfrak{S}_p}\sum_{i_1,\ldots,i_p=1}^Ny_{i_{\pi(1)},\ldots,i_{\pi(p)}}\tau_{i_{\pi(1)}}\cdots\tau_{i_{\pi(p)}}\\
        %	&=\sum_{i_1,\ldots,i_p=1}^Ny_{i_1,\ldots,i_p}\tau_{i_1}\cdots\tau_{i_p}\\
        %	&=\bigl\la Y,\tau^{\otimes p}\bigr\ra.
        %	\end{align*}
        %\end{proof}
        
        Recall the free energy $F_N(\beta)$ and the Rademacher prior $u$ from Section \ref{sec2}. Define an auxiliary free energy of the pure $p$-spin model with a Curie-Weiss type interaction as
        \begin{align}\label{af}
        AF_{N}(\beta)&=\frac{1}{N}\log\sum_{\sigma\in\Sigma_N}\frac{1}{2^N} \exp \Bigl(\frac{\beta}{\sqrt{2}}H_N(\sigma)+\frac{\beta^2}{2} N\Bigl(\frac{1}{N}\sum_{i=1}^Nh_i\sigma_i\Bigr)^p\Bigr),
        \end{align}
        where $h_i:=\sqrt{N}u_i.$ In Appendix, it will be established that the limit of $AF_N$ converges a.s. to a nonrandom quantity for any $\beta\geq 0$. Denote this limit by $AF(\beta).$ The following lemma relates the total variation distance between $W$ and $T$ to the free energy $F_N$ and the auxiliary free energy $AF_N$.
        
        \begin{lemma}
        	\label{thm7:lem3} For any $\beta \geq 0,$ we have that
        	\begin{align}
        	\begin{split}\label{thm7:lem3:eq1}
        	d_{TV}(W,T)&=\int_0^1\p\Bigl(F_N(\beta )<\frac{\beta ^2}{4}+\frac{\log x}{N}\Bigr)dx
        	\end{split}\\ 
        	\begin{split}
        	\label{thm7:lem3:eq2}
        	&=\int_0^1\p\Bigl(AF_{N}(\beta )>\frac{\beta ^2}{4}-\frac{\log x}{N}\Bigr)dx.
        	\end{split}
        	\end{align}
        \end{lemma}

        \begin{proof}
        	Note that $W$ has density $f_W(w)=\exp \bigl(-N\la w,w\ra/4\bigr)/C$ on $\Omega_N^s$ for some normalizing constant $C>0.$ For any $A\in \mathscr{B}(\Omega_N^s)$, a change of variables gives
        	\begin{align*}
        	\p\bigl(T\in A\bigr)&=\e_u\bigl[\p_W\bigl(W\in A-\beta  u^{\otimes p}\bigr)\bigr]\\
        	&=\e_u\bigl[\int_{A-\beta  u^{\otimes p}}f_W(w)dw\bigr]\\
        	&=\e_u\bigl[\int_Af_W(w-\beta  u^{\otimes p})dw\bigr]=\int_A \e_uf_W(w-\beta  u^{\otimes p})dw,
        	\end{align*} 
        	where $\e_u$ is the expectation with respect to $u$ only and $dw$ is the Lebesgue on $\Omega_N^s.$
        	This implies that the density of $T=W+\beta  u^{\otimes p}$ is given by $f_T(w):=\e_uf_W(w-\beta  u^{\otimes p}).$ Now since
        	\begin{align*}
        	f_T(w)&=\frac{1}{C}\e_u\exp\frac{-N}{4}\bigl\la w-\beta  u^{\otimes p},w-\beta  u^{\otimes p}\bigr\ra\\
        	&=f_W(w)\e_u\exp \frac{N}{4}\bigl(2\beta  \bigl\la w,u^{\otimes p}\bigr\ra-\beta  ^2 \bigl\la u^{\otimes p},u^{\otimes p}\bigr\ra\bigr),
        	\end{align*}
        	we obtain
        	\begin{align}\label{thm7:lem3:proof:eq1}
        	\frac{f_T(w)}{f_W(w)}&=\e_u\exp \frac{N}{4}\bigl(2\beta  \bigl\la w,u^{\otimes p}\bigr\ra-\beta  ^2 \bigl\la u^{\otimes p},u^{\otimes p}\bigr\ra\bigr).
        	\end{align}
        	Here, since $\bigl\la W,\tau^{\otimes p}\bigr\ra=\bigl\la Y,\tau^{\otimes p}\bigr\ra$ for any $\tau\in\mathbb{R}^N$ and $\bigl\la u^{\otimes p},u^{\otimes p}\bigr\ra=1$, we see that
        	\begin{align*}
        	\log \frac{f_T(W)}{f_W(W)}&=\log \e_u\exp \frac{N}{4}\bigl(2\beta  \bigl\la Y,u^{\otimes p}\bigr\ra-\beta  ^2 \bigr)\stackrel{d}{=}-\frac{\beta ^2N}{4}+NF_N(\beta )
        	\end{align*}
        	and
        	\begin{align*}
        	\log \frac{f_T(T)}{f_W(T)}&=\log \e_{u'}\exp \frac{N}{4}\bigl(2\beta  \bigl\la Y+\beta  u^{\otimes p},{u'}^{\otimes p}\bigr\ra-\beta  ^2 \bigr)\stackrel{d}{=}-\frac{\beta ^2N}{4}+NAF_{N}(\beta ),
        	\end{align*}
        	where $\e_{u'}$ is the expectation of $u'$, an independent copy of $u$ and independent of $W$, and the second equality of both displays used the assumption that $y_{i_1,\ldots,i_p}\stackrel{d}{=}g_{i_1,\ldots,i_p}\sqrt{2/N}$ and $\sqrt{N}S_N=\Sigma_N.$
        	Our proof is then completed by applying Lemma \ref{thm7:lem1}.
        \end{proof}
    
    \begin{remark}\label{rmk-1}\rm
    Aizenman, Lebowitz, and Ruelle \cite{ALR} showed that $N(F_N(\beta)-\beta^2/4)$ converges to a Gaussian central limit theorem. From \eqref{thm7:lem3:eq1}, one immediately sees that $\lim_{N\rightarrow\infty}d_{TV}(W,T)\in (0,1).$
    \end{remark}

    \subsection{Proof of Theorem \ref{thm7}$(i)$}    
        
        The central ingredient throughout our proof is played by the high temperature behavior of the pure $p$-spin model stated in Section~\ref{sec4} below, namely, a tight upper bound for the fluctuation of the free energy in Proposition \ref{prop1} and a good moment control for the concentration of the overlap $R_{1,2}$ around zero under the Gibbs measure in Theorem \ref{extra:thm3}. The former will be directly used to show that the total variation distance $d_{TV}(W,T)$ vanishes via the exact expression \eqref{thm7:lem3:eq1}, while the latter is vital in order to establish the impossibility of weak recovery of $u.$

        \begin{proof}[\bf Proof of Theorem \ref{thm7}$(i)$: Indistinguishability]
        	Let $p\geq 3$. Assume that $0<\beta  <\beta_p $.
        	For any $0<\varepsilon<1$, writing $\int_0^1=\int_0^{1-\varepsilon}+\int_{1-\varepsilon}^1$ in \eqref{thm7:lem3:eq1} gives 
        	\begin{align*}
        	d_{TV}(W,T)&\leq \varepsilon+\int_0^{1-\varepsilon}\p\Bigl(F_N(\beta )<\frac{\beta ^2}{4}+\frac{\log x}{N}\Bigr)dx\\
        	&\leq \varepsilon+\p\Bigl(F_N(\beta )<\frac{\beta ^2}{4}-\frac{\log  (1-\varepsilon)^{-1}}{N}\Bigr)\\
        	&\leq \varepsilon+\p\Bigl(\Bigl|F_N(\beta )-\frac{\beta ^2}{4}\Bigr|\geq \frac{\log  (1-\varepsilon)^{-1}}{N}\Bigr).
        	\end{align*}
        	To complete the proof, we use a key property about the fluctuation of the free energy stated in Proposition \ref{prop1} below, which says that there exists a constant $K$ such that 
        	$$
        	\p\Bigl(\Bigl|F_N(\beta )-\frac{\beta ^2}{4}\Bigr|\geq \frac{\log  (1-\varepsilon)^{-1}}{N}\Bigr)\leq \frac{K}{\bigl(\log (1-\varepsilon)^{-1}\bigr)^{2}N^{\frac{p}{2}-1}}
        	$$
        	for all $N\geq 1.$ From this, 
        	\begin{align}\label{add:eq-1}
        	d_{TV}(W,T)\leq \varepsilon+\frac{K}{\bigl(\log (1-\varepsilon)^{-1}\bigr)^{2}N^{\frac{p}{2}-1}}.
        	\end{align}
        	Since $p\geq 3,$ sending $N\rightarrow \infty$ and then $\varepsilon\downarrow 0$ implies that $W$ and $T$ are indistinguishable. 
        	\end{proof}

        \begin{remark}\label{rmk0}\rm
        	Take $\varepsilon=N^{-\delta}$ for $\delta=(p/2-1)/3$  and use $\log (1-\varepsilon)^{-1}\geq \varepsilon$ in \eqref{add:eq-1}. We obtain the rate of convergence,
        	\begin{align*}
        	d_{TV}(W,T)&\leq \varepsilon+\frac{K}{\varepsilon^2N^{\frac{p}{2}-1}}=\frac{1}{N^{\delta}}+\frac{K}{N^{\frac{p}{2}-1-2\delta}}=\frac{1+K}{N^{\frac{1}{3}\bigl(\frac{p}{2}-1\bigr)}}.
        	\end{align*}
        \end{remark}

    Next we continue to show that weak recovery of $u$ is impossible. For $\beta>0$, define a random probability measure on $\Omega_N^s\times \mathscr{B}(S_N)$ by
    \begin{align}\label{add:eq-8}
    \nu_{N,\beta}(w,A)&=\frac{\e_{u'}\bigl[\exp \frac{\beta N}{2}\la w,{u'}^{\otimes p}\ra;A\bigr]}{\e_{u'}\bigl[\exp \frac{\beta N}{2}\la w,{u'}^{\otimes p}\ra\bigr]}
    \end{align}
    for any $w\in \Omega_N^s$ and $A\in \mathscr{B}(S_N)$,
    where $\e_{u'}$ is the expectation with respect to $u'$, an independent copy of $u$. The following lemma relates the expectation of $(u,T)$ to $W$ by a change of measure in terms of $\e\nu_{N,\beta}.$
    
    \begin{lemma}\label{add:lem2}
    	Let $\zeta_N $ be a measurable function from $S_N\times \Omega_N^s$ to $[0,1].$ If $W$ and $T$ are indistinguishable, then
    	\begin{align*}
    	\lim_{N\rightarrow\infty}\Bigl|\e \zeta_N(u,T)-\e \int_{S_N} \zeta_N(\tau,W)\nu_{N,\beta}(W,d\tau)\Bigr|=0.
    	\end{align*}
    \end{lemma}
    
    \begin{proof}
    	Recall the densities $f_W$ and $f_T$ of $W$ and $T$ from Lemma \ref{thm7:lem3}. Let $f_u(\tau)$ be the probability mass function for $u.$ Since $u$ is independent of $W$, the joint density of $(u,T)$ is given by
    	$$
    	f_u(\tau)f_W(w-\beta\tau^{\otimes p}).
    	$$
    	This implies that
    	$\e[\zeta_N(u,T)|T]=\zeta_N(T),$
    	where
    	\begin{align*}
    	\zeta_N(w):=\frac{\sum_{\tau\in S_N}\zeta_N(\tau,w)f_u(\tau)f_W(w-\beta\tau^{\otimes p})}{f_T(w)}.
    	\end{align*}
    	Note that $0\leq\zeta_N(w)\leq 1$ since $0\leq\zeta_N(\tau,w)\leq 1$. For any $k\geq 1,$ define 
    	\begin{align*}
    	\phi_k(s)&=\left\{
    	\begin{array}{ll}
    	\frac{i}{k},&\mbox{if $s\in A_{k,i}$ for some $1\leq i\leq k-1$,}\\
    	1,&\mbox{if $s\in A_{k,k},$}
    	\end{array}\right.
    	\end{align*}
    	where $A_{k,i}:=[(i-1)/k,i/k)$ for $1\leq i\leq k-1$ and $A_{k,k}:=[(k-1)/k,1].$
    	Observe that $|\phi_k(s)-s|\leq 1/k$ for $s\in [0,1].$ From this and the triangle inequality,
    	\begin{align*}
    	&|\e \zeta_N(T)-\e \zeta_N(W)|\\
    	&\leq |\e \zeta_N(T)-\e\phi_k(\zeta_N(T))|+|\e\phi_k(\zeta_N(T))-\e\phi_k(\zeta_N(W))|+|\e \zeta_N(W)-\e\phi_k(\zeta_N(W))|\\
    	&\leq \frac{2}{k}+\sum_{i=1}^{k}\frac{i}{k}\bigl|\p\bigl(\zeta_N(T)\in A_{k,i}\bigr)-\p\bigl(\zeta_N(W)\in A_{k,i}\bigr)\bigr|.
    	\end{align*} 
    	Since $d_{TV}(W,T)$ converges to zero, each term in the above sum must vanish in the limit and thus, letting $N\rightarrow\infty$ and then $k\rightarrow\infty$ yields
    	\begin{align}\label{add:lem2:proof:eq1}
    	\lim_{N\rightarrow\infty}|\e \zeta_N(T)-\e \zeta_N(W)|=0.
    	\end{align} 
    	Now, write
    	\begin{align*}
    	\e \zeta_N(W)&=\int_{\Omega_N^s}\sum_{\tau\in S_N}\zeta_N(\tau,w)\frac{f_u(\tau)f_W(w-\beta\tau^{\otimes p})f_W(w)}{f_T(w)}dw.
    	\end{align*}
    	Note that $f_W(w)=\exp \bigl(-N\la w,w\ra/4\bigr)/C$ for some normalizing constant. Since from \eqref{thm7:lem3:proof:eq1} and $\la \tau^{\otimes p},\tau^{\otimes p}\ra=\la {u'}^{\otimes p},{u'}^{\otimes}\ra= 1$,
    	\begin{align*}
    	\frac{f_u(\tau)f_W(w-\beta\tau^{\otimes p})f_W(w)}{f_T(w)}&=\frac{\exp\frac{N}{4}\bigl(-\la w,w\ra+2\beta\la w,\tau^{\otimes p}\ra-\beta^2\la \tau^{\otimes p},\tau^{\otimes p}\ra\bigr) }{C2^N\e_{u'}\bigl[\exp \frac{N}{4}\bigl(2\beta\la w,{u'}^{\otimes p}\ra-\beta^2\la {u'}^{\otimes p},{u'}^{\otimes p}\ra\bigr)\bigr]}\\
    	&=f_W(w)\nu_{N,\beta}(w,\tau),
    	\end{align*}
    	it follows that
    	\begin{align*}
    	\e \zeta_N(W)&=\e \int_{S_N}\zeta_N(\tau,W)\nu_{N,\beta}(W,d\tau).
    	\end{align*}
    	From this and \eqref{add:lem2:proof:eq1}, the announced result follows.
    \end{proof}

    \begin{proof}[\bf Proof of Theorem \ref{thm7}$(i)$: Impossibility of weak recovery]  Let $p\geq 3$ and $0<\beta<\beta_p$. Let $\mu_N$ be a random probability measure on $\Omega_N^s\times \mathscr{B}(S_N)$ (see Definition \ref{def2}) and $c>0$. Our goal is to show that
    	\begin{align}\label{add:eq-10}
    	\lim_{N\rightarrow\infty}\p\Bigl(\int_{S_N}|\la u,\tau\ra|\mu_N(T,d\tau)\geq c\Bigr)=0.
    	\end{align}
    	Set
    	\begin{align*}
    	\zeta_N(\tau,w)&=\int_{S_N}\la \tau,\tau'\ra^2\mu_N(w,d\tau')
    	\end{align*}
    	for $(\tau,w)\in S_N\times \Omega_N^s.$ Note that $\zeta_N\in [0,1]$ and $\zeta_N$ is measurable.
    	From Lemma \ref{add:lem2},
    	\begin{align}\label{add:eq-9}
    	\lim_{N\rightarrow\infty}\Bigl|\e \int_{S_N}\la u,\tau'\ra^2\mu_N(T,d\tau')-\e\int_{S_N\times S_N}\la \tau,\tau'\ra^2\mu_N(W,d\tau')\nu_{N,\beta}(W,d\tau)\Bigr|=0.
    	\end{align}
    	We claim that the second expectation converges to zero. For notation convenience, we simply denote $\mu_N(d\tau')=\mu_N(W,d\tau')$ and $\nu_{N}(d\tau)=\nu_{N,\beta}(W,d\tau)$. Note that $$\la \tau,\tau'\ra^2=\sum_{i_1,i_2=1}^N\tau_{i_1}\tau_{i_2}\tau_{i_1}'\tau_{i_2}'.$$ The second term in the above equation can be controlled by
    	\begin{align*}
        &\e\int_{S_N\times S_N}\la \tau,\tau'\ra^2\mu_N(d\tau')\nu_{N}(d\tau)\\
    	&=\e\sum_{i_1,i_2=1}^N\int_{S_N\times S_N}\tau_{i_1}\tau_{i_2}\tau_{i_1}'\tau_{i_2}'\mu_N(d\tau')\nu_{N}(d\tau)\\
    	&=\sum_{i_1,i_2=1}^N\e\int_{S_N}\tau_{i_1}\tau_{i_2}\nu_{N}(d\tau)\cdot\int_{S_N}\tau_{i_1}'\tau_{i_2}'\mu_N(d\tau')\\
    	&\leq\sum_{i_1,i_2=1}^N\Bigl(\e\Bigl(\int_{S_N}\tau_{i_1}\tau_{i_2}\nu_{N}(d\tau)\Bigr)^2\Bigr)^{1/2}\cdot\Bigl(\e \Bigl(\int_{S_N}\tau_{i_1}'\tau_{i_2}'\mu_N(d\tau')\Bigr)^2\Bigr)^{1/2},
    	\end{align*}
    	where the last inequality used the Cauchy-Schwarz inequality. 
    	Using the Cauchy-Schwarz inequality again, the last inequality is bounded above by
    	\begin{align*}
    	&\Bigl(\sum_{i_1,i_2=1}^N\e\Bigl(\int_{S_N}\tau_{i_1}\tau_{i_2}\nu_{N}(d\tau)\Bigr)^2\Bigr)^{1/2}\cdot\Bigl(\sum_{i_1,i_2=1}^N\e \Bigl(\int_{S_N}\tau_{i_1}'\tau_{i_2}'\mu_N(d\tau')\Bigr)^2\Bigr)^{1/2}\\
    	&= \Bigl(\sum_{i_1,i_2=1}^N\e\int_{S_N\times S_N}\tau_{i_1}\tau_{i_2}\hat\tau_{i_1}\hat\tau_{i_2}\nu_{N}(d\tau)\nu_N(d\hat{\tau})\Bigr)^{1/2}\\
    	&\quad\cdot\Bigl(\sum_{i_1,i_2=1}^N\e \int_{S_N\times S_N}\tau_{i_1}'\tau_{i_2}'\hat\tau_{i_1}'\hat \tau_{i_2}'\mu_N(d\tau')\mu_N(d\hat\tau')\Bigr)^{1/2}\\
    	&=\Bigl(\e\int_{S_N\times S_N}\la \tau,\hat\tau \ra^2\nu_{N}(d\tau)\nu_N(d\hat{\tau})\Bigr)^{1/2}\cdot\Bigl(\e\int_{S_N\times S_N}\la\tau',\hat \tau'\ra^2\mu_N(d\tau')\mu_N(d\hat\tau')\Bigr)^{1/2}.
    	\end{align*}
    	Here the second bracket is bounded above by $1$. As for the first one, we observe that $\nu_N$ is in distribution equal to the Gibbs measure $G_{N,\beta}$ defined in \eqref{gibbs} and if we write $\sigma^1=\sqrt{N}\tau$ and $\sigma^2=\sqrt{N}\tau'$, then in distribution, $\sigma^1,\sigma^2$ are independent samplings from $G_{N,\beta}$ and $\la \tau,\hat\tau \ra$ is the overlap $R_{1,2}$ between $\sigma^1$ and $\sigma^2$. As a result,
    	\begin{align*}
    	\e\int_{S_N\times S_N}\la \tau,\hat\tau \ra^2\nu_{N}(d\tau)\nu_N(d\hat{\tau})=\e \la R_{1,2}^2\ra_{\beta},
    	\end{align*}
    	where $\la \cdot\ra_\beta$ is the Gibbs average with respect to the product measure $G_{N,\beta}\times G_{N,\beta}$. Now, since $0<\beta<\beta_p,$ we can apply Theorem \ref{extra:thm3} to control the right-hand side by the bound $$
    	\e \la R_{1,2}^2\ra_{\beta}\leq \frac{K}{N}
    	$$
    	for some constant $K$ independent of $N.$ Hence, from the above inequalities, $$
    	\lim_{N\rightarrow\infty}\e\int_{S_N\times S_N}\la \tau,\tau'\ra^2\mu_N(d\tau')\nu_{N}(d\tau)=0.$$
    	From \eqref{add:eq-9},
    	$$
    	\lim_{N\rightarrow\infty}\e \int_{S_N}\la u,\tau'\ra^2\mu_N(T,d\tau')=0,
    	$$
    	which gives the desired limit \eqref{add:eq-10} by using Markov's and Jensen's inequalities.
    \end{proof}

    \subsection{Proof of Theorem \ref{thm7}$(ii)$}
    
    While we have seen that the high temperature behavior of the pure $p$-spin model has been of great use in obtaining Theorem \ref{thm7}$(i)$, the proof of Theorem \ref{thm7}$(ii)$ below relies only on the low temperature behavior of the free energies, that is, $F(\beta)<\beta^2/4$ and $AF(\beta)>\beta^2/4$ for $\beta>\beta_p.$ The proof is relatively simpler than that for Theorem \ref{thm7}$(i).$
    
    \begin{proof}[\bf Proof of Theorem \ref{thm7}$(ii)$: Distinguishability]
    	Let $p\geq 3$. Assume that $\beta >\beta _p.$ Since $F(\beta )<\beta ^2/4$ and $F_N(\beta )-\log x/N$ converges to $F(\beta )$ a.s., 
    	\begin{align*}
    	\lim_{N\rightarrow\infty}\p\Bigl(F_N(\beta )<\frac{\beta ^2}{4} +\frac{\log x}{N}\Bigr)=\p\Bigl(F(\beta )\leq \frac{\beta ^2}{4} \Bigr)=1.
    	\end{align*}
    	Thus, from \eqref{thm7:lem3:eq1} and the dominated convergence theorem, $W$ and $T$ are distinguishable.
    \end{proof}

          Next we show that weak recovery of $u$ is possible. Recall $h_i$ from the definition of $AF_N(\beta).$ Let $u'$ be an independent copy of $u$ and be independent of $Y.$ For fixed $\beta>0$, define an interpolating free energy between $F_N(\beta)$ and $AF_N(\beta)$ by
        \begin{align}
        \begin{split}\label{af2}
       L_{N}(x)&=\frac{1}{N}\log \e_{u'} \exp \frac{\beta  N}{2}\bigl<Y+xu^{\otimes p},{u'}^{\otimes p}\bigr>\\
       &\stackrel{d}{=}\frac{1}{N}\log \sum_{\sigma\in\Sigma_N}\frac{1}{2^N}\exp\Bigl(\frac{\beta }{\sqrt{2}}H_N(\sigma)+\frac{\beta  x}{2}\Bigl(\frac{1}{N}\sum_{i=1}^Nh_i\sigma_i\Bigr)^p\Bigr)
       \end{split}
       \end{align}
        for $x>0,$ where $\e_{u'}$ is the expectation with respect to $u'$ only. Note that in distribution $L_N(0)=F_N(\beta)$ and $L_N(\beta)=AF_N(\beta)$. Similar to $AF_N(\beta)$, it can be shown that $L_N$ also converges to a nonrandom quantity for any $x>0,$ see Proposition \ref{add:prop1} in Appendix. Denote this limit by $L.$ From now on, we use $D_+f$ and $D_-f$ to denote the right and left derivatives of $f$ whenever they exist. Note that since $L$ is convex, $D_-L$ exists everywhere. Recall $\nu_{N,\beta}$ from \eqref{add:eq-8}.
        
        \begin{lemma}\label{add:lem1}
       	Let $\beta>0.$  For any $\varepsilon>0,$ we have 
       	\begin{align*}
       	\lim_{N\rightarrow\infty}\p\Bigl(\int_{S_N}|\la u,\tau\ra|^p\nu_{N,\beta}(T,d\tau)\geq D_-L (\beta)-\varepsilon\Bigr)=1.
       	\end{align*}
       \end{lemma} 
       
       \begin{proof}
       	Note that $\bigl<u^{\otimes p},{u'}^{\otimes p}\bigr>=\la u,u'\ra^p.$ For any $\eta>0,$ 
       	\begin{align*}
       	&\frac{1}{N}\log \e_{u'}\Bigl[ \exp \frac{\beta  N}{2}\bigl<Y+\beta  u^{\otimes p},{u'}^{\otimes p}\bigr>;\la u,u'\ra^p\leq D_-L (\beta )-\varepsilon \Bigr]\\
       	&\leq \frac{1}{N}\log \e_{u'}\Bigl[ \exp \frac{\beta  N}{2}\bigl(\bigl<Y,{u'}^{\otimes p}\bigr\ra+(\beta -\eta)\bigl\la u^{\otimes p},{u'}^{\otimes p}\bigr>\bigr) \Bigr]+\eta\bigl(D_-L (\beta )-\varepsilon\bigr)\\
       	&=L_N (\beta -\eta)+\eta\bigl(D_-L (\beta )-\varepsilon\bigr).
       	\end{align*}
       	To control the last inequality, write
       	\begin{align*}
       	L_N (\beta -\eta)+\eta\bigl(D_-L (\beta )-\varepsilon\bigr)&=\eta\Bigl(D_-L (\beta )-\frac{L_N (\beta)-L_N (\beta-\eta )}{\eta}\Bigr)+\bigl(L_N (\beta)-\eta\varepsilon\bigr)
       	\end{align*}
       	and pass to limit
       	\begin{align*}
       	\lim_{N\rightarrow\infty}\Bigl(L_N (\beta -\eta)+\eta\bigl(D_-L (\beta )-\varepsilon\bigr)\Bigr)&=\eta\Bigl(D_-L (\beta )-\frac{L (\beta)-L (\beta-\eta )}{\eta}\Bigr)+\bigl(L (\beta)-\eta\varepsilon\bigr).
       	\end{align*}
       	Here, by the left-differentiability of $L ,$ the first bracket on the right-hand side converges to zero as $\eta\downarrow 0.$ Thus, we can choose $\eta$ small enough such that the right hand side is controlled by $L (\beta)-\eta\varepsilon/2.$ Consequently, from the sub-Gaussian concentration inequality for $L_N $ (see \cite[Proposition 9]{Chen14}), we see that there exists a positive constant $K$ independent of $N$ such that with probability at least $1-Ke^{-N/K}$,
       	\begin{align*}
       	\frac{1}{N}\log \e_{u'}\Bigl[ \exp \frac{\beta  N}{2}\bigl<T,{u'}^{\otimes p}\bigr>;\la u,u'\ra^p\leq D_-L (\beta )-\varepsilon \Bigr]\leq L_N (\beta )-\frac{\eta\varepsilon}{4}
       	\end{align*} 			
       	and this implies that
       	\begin{align*}
       	\nu_{N,\beta}\Bigl(T,\bigl\{\tau\in S_N\big| \la u,\tau\ra^p\leq D_-L (\beta)-\varepsilon\bigr\}\Bigr)\leq e^{-\frac{\eta\varepsilon N}{4}}.
       	\end{align*}
       	As a result, the assertion follows by
       	\begin{align*}
       	\p\Bigl(\int_{S_N}\la u,\tau\ra^p\nu_{N,\beta}(T,d\tau)\geq \bigl(D_-L (\beta)-\varepsilon\bigr)\bigl(1-e^{-\frac{\eta\varepsilon N}{4}}\bigr)-e^{-\frac{\eta\varepsilon N}{4}}\Bigr)\geq 1-Ke^{-\frac{N}{K}}
       	\end{align*}
       	and noting that $\int_{S_N}\la u,\tau\ra^p\nu_{N,\beta}(T,d\tau)\leq \int_{S_N}|\la u,\tau\ra|^p\nu_{N,\beta}(T,d\tau)$.
       \end{proof}

        	\begin{proof}[\bf Proof of Theorem \ref{thm7}$(ii)$: Possibility of weak recovery]
        		Let $p\geq 3$. Assume that $\beta >\beta _p.$
        		First we claim that
        		\begin{align}\label{extra:thm7:proof:eq1}
        		AF(\beta )\geq \frac{\beta ^2}{4}. 
        		\end{align}
        		Assume on the contrary that $AF(\beta )<\beta ^2/4$. Since $AF_N(\beta )+\log x/N$ converges to $AF(\beta )$ a.s., we have
        		\begin{align*}
        		\p\Bigl(AF_{N}(\beta )>\frac{\beta ^2}{4}-\frac{\log x}{N}\Bigr)\rightarrow \p\Bigl(AF(\beta )>\frac{\beta ^2}{4}\Bigr)=0.
        		\end{align*}
        		This and \eqref{thm7:lem3:eq2} together leads to a contradiction,
        		\begin{align*}
        		1=\lim_{N\rightarrow\infty}d_{TV}(W,T)&=\int_0^1\lim_{N\rightarrow\infty}\p\Bigl(AF_N(\beta )>\frac{\beta ^2}{4}-\frac{\log x}{N}\Bigr)dx=0.
        		\end{align*}
        		Thus, \eqref{extra:thm7:proof:eq1} must be valid.

        	To show that weak recovery is possible, observe that since  $L_{N}(\beta )=AF_N(\beta )$ and $L_{N}(0)=F_N(\beta )$ in distribution, it follows from \eqref{extra:thm7:proof:eq1} that 
        	$$
        	L(\beta )=AF(\beta )\geq \frac{\beta ^2}{4}>F(\beta )=L(0).
        	$$
        	Since $L$ is convex, there exists a point $x_0\in (0,\beta )$ such that $D_-L(x_0)>0$. Indeed, if not $L$ will be a constant function on $(0,\beta)$, a contradiction. Now using the convexity of $L$ again gives $D_-L(\beta)\geq D_-L(x_0)>0.$ This and Lemma \ref{add:lem1} together complete our proof by letting $\mu_N(w,\tau)=\nu_{N,\beta}(w,\tau)$ and $c=D_-L(\beta)/2$ and noting that $\int_{S_N}|\la u,\tau\ra|^p\nu_{N,\beta}(T,d\tau)\leq \int_{S_N}|\la u,\tau\ra|\nu_{N,\beta}(T,d\tau)$.
        \end{proof}

\section{The pure $p$-spin model}\label{sec4}

 Recall the pure $p$-spin Hamiltonian $H_N$ and the high temperature regime $\mathcal{R}$ from Section \ref{sec2}. The aim of this section is to establish a complete description of the high temperature behavior of the pure $p$-spin model.
 % via the cavity method and Guerra's replica symmetry breaking scheme, originated from the study of spin glass models,

 \subsection{High temperature behavior} 
 As mentioned before, the limiting free energy $F_N(\beta)$ converges to a nonrandom quantity $F(\beta)$. This quantity can also be expressed in terms of the famous Parisi formula, which we state as follows. Throughout the rest of the paper, we set
 $$\xi(s)=\frac{s^p}{2}$$ for $s\in[0,1]$. Let $\mathcal{M}$ be the collection of all cumulative distribution functions on $[0,1]$ equipped with the $L^1$ distance with respect to the Lebesgue measure. This is usually called the space of functional order parameters in physics. For $\beta>0,$ define a functional $\mathcal{P}_\beta$ on $\mathcal{M}$ by
 \begin{align*}
 \mathcal{P}_\beta(\alpha)&=\Phi_{\beta,\alpha}(0,0)-\frac{\beta^2}{2}\int_0^1\alpha(s)\xi''(s)sds,
 \end{align*}
 where $\Phi_{\beta,\alpha}$ is the weak solution \cite{JT2} to the following PDE,
 \begin{align*}
 \partial_t\Phi_{\beta,\alpha}&=-\frac{\beta^2\xi''}{2}\Bigl(\partial_{xx}\Phi_{\beta,\alpha}+\alpha\bigl(\partial_x\Phi_{\beta,\alpha}\bigr)^2\Bigr)
 \end{align*}
 with boundary condition $\Phi_{\beta,\alpha}(1,x)=\log \cosh x.$ For any $\beta>0,$ the Parisi formula states that
 \begin{align*}
 \lim_{N\rightarrow\infty}F_N(\beta)&=\inf_{\alpha\in \mathcal{M}}\mathcal{P}_\beta(\alpha).
 \end{align*}
 Although we only consider the pure $p$-spin model here, this formula also holds in more general setting. 
 Indeed, Talagrand \cite{Tal03} established the Parisi formula in the case of the mixed even $p$-spin models. Later Panchenko \cite{Pan00} extends its validity to general mixtures of the model.
 Recently, it was understood by Auffinger and Chen \cite{AC14} that the functional $\mathcal{P}_\beta$ is strictly convex, which guarantees the uniqueness of the minimizer for $\mathcal{P}_\beta$. We shall call this minimizer the Parisi measure and denote it by $\alpha_P.$ 
 
 Recall that the high temperature regime $\mathcal{R}$ of $H_N$ is the collection of all $\beta>0$ that satisfy $F(\beta)=\beta^2/4.$ By a direct computation, the validity of this equation is the same as saying that the Parisi measure satisfies $\alpha_P\equiv 1$, concluding from the uniqueness of the Parisi measure. This case is usually called the {\it replica symmetric} solution of the model in physics literature \cite{MPV}. 
 %In other words, the high temperature $\mathcal{R}$ is the collection of all $\beta$ that exhibit replica symmetric solution. 
 Recall that the critical temperature $\beta_p $ is defined as the maximum of $\mathcal{R}$.  Let $g$ be a standard normal random variable. For $\beta>0,$ define an auxiliary function by
 \begin{align*}
 \rho_\beta(s)&:=\e \tanh^2\bigl(\beta g\sqrt{\xi'(s)}\bigr)\cosh\bigl(\beta g\sqrt{\xi'(s)}\bigr)e^{-\frac{\beta^2}{2}\xi'(s)},\,\,\forall s\in[0,1].
 \end{align*}
 Our first theorem provides one way to characterize $\mathcal{R}$ and $\beta_p .$ 
 
  \begin{theorem}\label{extra:thm1} For any $p\geq 2,$ $\mathcal{R}=(0,\beta_p ].$ In addition, the following two statements hold:
  	
  	\begin{itemize}
  		
  		\item[$(i)$] Let $\beta>0.$ Then $\beta\in \mathcal{R}$ if and only if 
  		\begin{align*}
  		%\label{extra:thm1:eq2}
  		\int_0^r\xi''(s)\bigl(\rho_\beta(s)-s\bigr)ds\leq 0,\,\,\forall r\in (0,1].
  		\end{align*} 
  		\item[$(ii)$] If $0<\beta<\beta_p $, then
  		\begin{align}\label{extra:thm1:eq1}
  		\int_0^r\xi''(s)\bigl(\rho_\beta(s)-s\bigr)ds<0,\,\,\forall r\in (0,1].
  		\end{align} 
  	\end{itemize}
  \end{theorem}

  Item $(i)$ is essentially the first order optimality condition in order to obtain the replica symmetric solution. Item $(ii)$ states that the replica symmetric solution is stable if $\beta$ stays away from the criticality. This is the most crucial property that will allow us to establish the desired high temperature behavior of the overlap all the way up to the critical temperature. 
  
  Define the Gibbs measure by
  \begin{align}\label{gibbs}
  G_{N,\beta}(\sigma)&=\frac{\exp \beta H_N(\sigma)}{\sum_{\sigma'\in\Sigma_N}\exp \beta H_N(\sigma')}.
  \end{align}
  For i.i.d. samplings $\sigma^1,\sigma^2$ from $G_{N,\beta}$, we use $\la \cdot\ra_\beta$ to denote the expectation with respect to the product measure $G_{N,\beta}^{\otimes 2}$. Recall the overlap $R_{1,2}$ between $\sigma^1$ and $\sigma^2$ from \eqref{op}. Our next two theorems show that the overlap is concentrated around $0$ in the high temperature regime with exponential tail probability and moment control.
  
    \begin{theorem}
  	\label{extra:thm2}
  	Assume that $p\geq 2$. Fix $0<a<b<\beta_p $. For any $\varepsilon>0,$ there exists a constant $K$ such that for any $\beta\in [a,b],$
  	\begin{align}\label{extra:thm2:eq1}
  	\e \bigl\la I\bigl(|R_{1,2}|\geq \varepsilon\bigr)\bigr\ra_\beta\leq K\exp\Bigl(-\frac{N}{K}\Bigr),\,\,\forall N\geq 1,
  	\end{align}
  	where $I$ is an indicator function.
  \end{theorem} 
 
 \begin{theorem}
 	\label{extra:thm3}
 	Assume that $p\geq 3$. Fix $0<b<\beta_p $. For any $k\geq 1,$ there exists a constant $K>0$ such that for any $\beta\in[0,b]$,
 	\begin{align*}
 	\e \bigl\la R_{1,2}^{2k}\bigr\ra_\beta\leq \frac{K}{N^k},\,\,\forall N\geq 1.
 	\end{align*}
 \end{theorem}
 
  In the case of the Sherrington-Kirkpatrick model $(p=2)$, it was computed that $\beta_2 =1$ (see Remark 2 in \cite{C14}) and the same results as Theorems \ref{extra:thm2} and \ref{extra:thm3} were obtained in Talagrand's book \cite[Chapters 11 and 13]{Talbook2}. As for $p\geq 3,$ Bardina, M\'{a}rquez, Rovira, and Tindel \cite{Tindel} established Theorem \ref{extra:thm3}  for some $b\ll\beta_p $ as $p$ increases. Our main contribution here is that the concentration of the overlap is valid up to the critical temperature. As an application of Theorem \ref{extra:thm3}, we deduce a control on the fluctuation of the free energy in high temperature regime.
  
  \begin{proposition}\label{prop1}
  Assume that $p\geq 3.$ Fix $0<b<\beta_p $. There exists a constant $K$ such that for any $0\leq \beta\leq b$, 
  \begin{align}\label{prop1:eq1}
  \p\Bigl(\Bigl|F_N(\beta)-\frac{\beta^2}{4}\Bigr|\geq r\Bigr)&\leq \frac{K}{r^2N^{p/2+1}}
  \end{align}
  for any $r>0$ and $N\geq 1.$	
  \end{proposition}

  This theorem basically says that the fluctuation of $F_N(\beta)$ is at most of the order $N^{-p/4-1/2}.$ Indeed, if there exists some $\delta_N\uparrow \infty$ such that
  \begin{align*}
  \p\Bigl(\Bigl|F_N(\beta)-\frac{\beta^2}{4}\Bigr|\geq \delta_NN^{-\frac{p}{4}-\frac{1}{2}}\Bigr)\geq c>0
  \end{align*}
  for all $N\geq 1,$ then this contradicts \eqref{prop1:eq1}. For $p=2,$ Aizenman, Lebowitz, and Ruelle \cite{ALR} proved that $N^{-p/4-1/2}=N^{-1}$ is the right order of the fluctuation for $F_N(\beta)$ and $N(F_N(\beta)-\beta^2/4)$ converges to a Gaussian random variable up to the critical temperature $\beta_p =1.$ Similarly, for even $p\geq 4,$ Bovier, Kurkova, and L\"owe \cite{BKL} also showed that $N^{p/4+1/2}(F_N(\beta)-\beta^2/4)$ has a Gaussian fluctuation up to certain temperature strictly less than $\beta_p .$ From these, it is tempting to conjecture that $N^{p/4+1/2}(F_N(\beta)-\beta^2/4)$ follows Gaussian law in the weak limit in the entire high temperature regime for all $p\geq 3.$ Based on Theorems \ref{extra:thm2} and \ref{extra:thm3}, this should be achievable by an adoption of the argument for the Sherrington-Kirkpatrick model \cite[Section 11.4]{Talbook2}. We do not pursue this direction here.
  
  \begin{remark}\label{rmk1}\rm
  	One can also consider the mixed $p$-spin model, i.e., the Hamiltonian $H_N$ is again a Gaussian process on $\Sigma_N$ with zero mean and covariance structure $\e H_N(\sigma^1)H_N(\sigma^2)=N\xi(R_{1,2})$ for some $\xi(s):=2^{-1}\sum_{p\geq 2}c_ps^p$ with $c_p\geq 0$ and $\sum_{p\geq 2}c_p=1.$ In a similar manner, one can define its free energy and high temperature regime as those for $F_N(\beta)$ and $\mathcal{R}.$ In this general setting, it can be checked that Theorem \ref{extra:thm1} remains valid. As for Theorems \ref{extra:thm2} and \ref{extra:thm3} and Proposition \ref{prop1}, they also hold as long as there exists some $p\geq 3$ such that $c_p\neq 0$ and $c_{p'}=0$ for all $2\leq p'<p.$ 
  \end{remark}

 \subsection{Proof of Theorem \ref{extra:thm1}}
 
 We now turn to the proof of Theorem \ref{extra:thm1} by solving the Parisi formula. First, we establish a key lemma on the strict monotonicity of $\rho_\beta$ in $\beta.$ This type of inequality was  initially discovered in \cite{AC15} in order to establish a Legendre duality principle for the Parisi formula. Surprisingly, we find that it also serves as one of the central ingredients in the description of the high temperature behavior of the model.  
 
 \begin{lemma}\label{extra:lem5}
 For any $s\in (0,1]$ and $0<\beta_1<\beta_2,$ we have $\rho_{\beta_1}(s)<\rho_{\beta_2}(s).$	
 \end{lemma}

\begin{proof} Let $s\in (0,1]$ and $0<\beta_1<\beta_2$ be fixed.  Denote 
\begin{align*}
A_1(x)&=\log \cosh(\beta_1 x\sqrt{\xi'(s)}),\\
A_2(x)&=\log \cosh(\beta_2 x\sqrt{\xi'(s)}),\\
C(x)&=\tanh^2(\beta_1x\sqrt{\xi'(s)}).
\end{align*}
Let $g'$ be standard normal independent of $g$. Set 
\begin{align*}
S(t,x)&=\frac{\exp \bigl((1-t)A_1(x)+tA_2(x)\bigr)}{\e \exp \bigl((1-t)A_1(g')+tA_2(g')\bigr)}.
\end{align*}	
Define  $f(t)=\e C(g)S(t,g).$ Note that $A_1,A_2,$ and $C$ are even. A direct computation gives
\begin{align*}
f'(t)&=\e C(g)\bigl(A_2(g)-A_1(g)\bigr)S(t,g)-\e C(g)S(t,g)\cdot \e\bigl(A_2(g')-A_1(g')\bigr)S(t,g')\\
&=\e C(|g|)\bigl(A_2(|g|)-A_1(|g|)\bigr)S(t,|g|)-\e C(g)S(t,|g|)\cdot \e\bigl(A_2(|g'|)-A_1(|g'|)\bigr)S(t,|g'|)\\
&=\e \bigl(C(|g|)-C(|g'|)\bigr)\bigl(\bigl(A_2(|g|)-A_1(|g|)\bigr)-\bigl(A_2(|g'|)-A_1(|g'|)\bigr)\bigr)S(t,|g|)S(t,|g'|).
\end{align*}
Since $\beta_2>\beta_1>0,$ this implies that $A_2-A_1$ and $C$ are strictly increasing on $[0,\infty).$ Therefore, with probability one,
\begin{align*}
\bigl(C(|g|)-C(|g'|)\bigr)\bigl(\bigl(A_2(|g|)-A_1(|g|)\bigr)-\bigl(A_2(|g'|)-A_1(|g'|)\bigr)\bigr)>0.
\end{align*}
 It follows that $f'(t)>0$ for $t\in (0,1)$ and as a result, noting that $\e \cosh(\beta\sqrt{\xi'(s)})=e^{\beta^2\xi'(s)/2}$ gives
\begin{align*}
\rho_{\beta_1}(s)&=f(0)<f(1)=\e \tanh^2(\beta_1g\sqrt{\xi'(s)})\cosh(\beta_2g\sqrt{\xi'(s)})e^{-\frac{\beta_2^2}{2}\xi'(s)}<\rho_{\beta_2}(s),
\end{align*}
where the last inequality used the fact that $\tanh^2(\beta_1x\sqrt{\xi'(s)})<\tanh^2(\beta_2x\sqrt{\xi'(s)})$ for all $x\neq 0$ since $\beta_1<\beta_2.$
\end{proof}

 \begin{proof}
 [\bf Proof of Theorem \ref{extra:thm1}] To prove $\mathcal{R}=(0,\beta_p ]$, we first notice two useful facts established in \eqref{extra:lem9:proof:eq1} and \eqref{extra:lem9:proof:eq3} below that 
 \begin{align*}
 \e F_N(\beta)-\frac{\beta^2}{4}&=\int_0^\beta l\e \bigl\la R_{1,2}^p\bigr\ra_ldl
 \end{align*}
 and
 \begin{align*}
 \e \bigl\la R_{1,2}^p\bigr\ra_l\geq 0.
 \end{align*} 
 Now if $F(\beta)=\beta^2/4$, then the first display implies
 \begin{align*}
 \lim_{N\rightarrow\infty}\int_0^\beta l\e \bigl\la R_{1,2}^p\bigr\ra_ldl=\lim_{N\rightarrow\infty}\Bigl(\e F_N(\beta)-\frac{\beta^2}{4}\Bigr)=0.
 \end{align*}
 Consequently, from the second display, for any $\beta'\in [0,\beta],$
 \begin{align*}
 \lim_{N\rightarrow\infty}\Bigl(\e F_N(\beta')-\frac{{\beta'}^2}{4}\Bigr)=\lim_{N\rightarrow\infty}\int_0^{\beta'} l\e \bigl\la R_{1,2}^p\bigr\ra_ldl=0.
 \end{align*}
 Thus, $(0,\beta]\subset\mathcal{R}$. This implies that $\mathcal{R}=(0,\beta_p ]$ by the definition of $\mathcal{R}.$
 
 Next we justify $(i)$ and $(ii).$ From \cite{C14}, one can compute the directional derivative of $\mathcal{P}_\beta$ using the stochastic optimal control theory. In particular, by applying \cite[Theorem 2]{C14} and taking $\alpha_0\equiv 1$, 
 \begin{align}\label{dir}
 D_+\mathcal{P}_\beta \bigl((1-\lambda)\alpha_0+\lambda\alpha\bigr)\Bigl|_{\lambda=0}&=\frac{\beta^2}{2}\int_0^1\xi''(s)\bigl(\alpha(s)-\alpha_0(s)\bigr)\bigl(\rho_\beta(s)-s\bigr)ds
 \end{align}
  for all $\alpha\in \mathcal{M}.$ 
 Write by using Fubini's theorem,
 \begin{align*}
 &\int_0^1\xi''(s)\bigl(\alpha(s)-\alpha_0(s)\bigr)\bigl(\rho_\beta(s)-s\bigr)ds\\
 &=\int_0^1\int_0^s\xi''(s)\bigl(\rho_\beta(s)-s\bigr)\alpha(dr)ds-\int_0^1\xi''(s)\alpha_0(s)\bigl(\rho_\beta(s)-s\bigr)ds\\
 &=\int_0^1\Bigl(\int_{r}^1\xi''(s)\bigl(\rho_\beta(s)-s\bigr)ds \Bigr)\alpha(dr)-\int_0^1\xi''(s)\bigl(\rho_\beta(s)-s\bigr)ds.
 \end{align*}
 Assume that $0<\beta\leq \beta_p .$ In this case, $\alpha_0$ is the Parisi measure and thus, the above equation is nonnegative for all $\alpha$ by the optimality of $\alpha_0.$ This implies 
 \begin{align*}
 \int_{r}^1\xi''(s)\bigl(\rho_\beta(s)-s\bigr)ds\geq \int_0^1\xi''(s)\bigl(\rho_\beta(s)-s\bigr)ds
 \end{align*}
 for all $r\in[0,1].$ Consequently, for all $r\in [0,1]$,
 \begin{align}\label{extra:thm1:proof:eq1}
 \int_{0}^r\xi''(s)\bigl(\rho_\beta(s)-s\bigr)ds\leq 0.
 \end{align}
 Conversely, if this inequality holds for some $\beta>0$, then we can argue in the backward order to see that the directional derivative \eqref{dir} of $\mathcal{P}_\beta$ at $\alpha_0$ is nonnegative for all $\alpha\in \mathcal{M}.$ Consequently, from the strict convexity of $\mathcal{P}_\beta$, $\alpha_0$ is a minimizer and it must be the Parisi measure. It then follows that $F(\beta)=\mathcal{P}_\beta(\alpha_0)=\beta^2/4$ and thus $\beta\in \mathcal{R}.$ This gives $(i).$ As for $(ii),$ note that $\beta_p $ also satisfies \eqref{extra:thm1:proof:eq1}. If there exists some $0<\beta<\beta_p$ such that $$
 \int_{0}^r\xi''(s)\bigl(\rho_\beta(s)-s\bigr)ds=0
 $$
 for some $r\in (0,1],$ then from Lemma~\ref{extra:lem5}, we arrive at a contradiction,
 \begin{align*}
 0&<\int_0^r\xi''(s)\bigl(\rho_{\beta_p}(s)-\rho_{\beta}(s)\bigr)ds\\
 &=\int_{0}^r\xi''(s)\bigl(\rho_{\beta_p}(s)-s\bigr)ds-\int_{0}^r\xi''(s)\bigl(\rho_\beta(s)-s\bigr)ds\\
 &=\int_{0}^r\xi''(s)\bigl(\rho_{\beta_p}(s)-s\bigr)ds\\
 &\leq 0
 \end{align*}
 Therefore, \eqref{extra:thm1:eq1} must valid.
 \end{proof}

 \subsection{Proof of Theorems \ref{extra:thm2} and \ref{extra:thm3}}
 
 First we give a brief description of our arguments for Theorems \ref{extra:thm2} and \ref{extra:thm3}. Let $B_N$ be the collection of all $v$ such that $v=R_{1,2}$ for some $\sigma^1,\sigma^2\in \Sigma_N.$ Let $0<\beta<\beta_p $ and $\varepsilon>0$ be fixed. For $v\in B_N$, consider the following coupled free energy
 \begin{align}\label{ineq1}
 CF_{N,v}(\beta)&:=\frac{1}{N}\log \sum_{R_{1,2}=v}\frac{1}{2^{2N}}\exp \frac{\beta}{\sqrt{2}}\bigl(H_N(\sigma^1)+H_N(\sigma^2)\bigr).
 \end{align}
 Our approach relies on finding a good bound for this coupled free energy so that as long as $v$ is away from zero, then there exists a constant $\delta>0$ such that 
 \begin{align}\label{add1:eq1}
 \e CF_{N,v}(\beta)\leq 2\e F_N(\beta)-\delta.
 \end{align}
 If this is valid, then an application of sub-Gaussian concentration inequality for the free energies (see \cite[Proposition 9]{Chen14}) implies that with overwhelmingly probability,
 $$
 CF_{N,v}(\beta)\leq F_N(\beta)-\frac{\delta}{2},
 $$
 which is equivalent to $\la I(R_{1,2}=v )\ra_\beta<e^{-\delta N/2}.$ This means that $R_{1,2}$ can not charge the value $v\neq 0$ and the exponential tail control allows us to deduce Theorem \ref{extra:thm2}. Based on Theorem \ref{extra:thm2}, we adopt the cavity method \cite{Talbook1,Talbook2} to obtain Theorem \ref{extra:thm3}. In order to obtain \eqref{add1:eq1}, we use the two-dimensional Guerra-Talagrand bound as in \cite[Proposition 14.6.3.]{Talbook2}. In the case of even $p,$ this bound is known to be valid and we show that the above argument can be carried through. However, for odd $p\geq 3$, the validity of the Guerra-Talagrand bound remains unknown. Surprisingly, if we consider $v\geq 0$ and restrict the functional order parameters to be of {\it one-step replica symmetry breaking}, it can be shown that this bound indeed holds, see the discussion in Subsection \ref{final}. As for $v<0,$ this case is relatively easy in that we can control \eqref{ineq1} by a direct use of Jensen's inequality to obtain a good upper bound for \eqref{ineq1}. Combining these two cases together, we show that \eqref{add1:eq1} is also achievable for odd $p\geq 3$ and this gives Theorems \ref{extra:thm2} and \ref{extra:thm3}.
  
 \subsubsection{Even $p$ case}
 
 \begin{proof}[\bf Proof of Theorem \ref{extra:thm2}: even $p\geq 2$]
 
 	Let $\varepsilon>0.$ Fix $v\in B_N$ with $v\geq\varepsilon$.  The two-dimensional version of the Guerra-Talagrand bound \cite[Chapter 15]{C14} states that if one chooses the same parametrization as \cite[Proposition 14.6.3]{Talbook2}, then the coupled free energy $CF_{N,v}(\beta)$ can be controlled by an upper bound of a similar expression as $\mathcal{P}_\beta(\alpha).$ More precisely, let $\mathcal{N}_v $ be the space of functional order parameters, the collection of all nonnegative and nondecreasing functions with right continuity and $\gamma(v)\leq 2$ on $[0,v]$. 
 	%Note that while every element $\alpha$ in $\mathcal{M}$ satisfies $\alpha(1)=1,$ the upper bound $2$ for $\gamma\in \mathcal{N}_v$ takes into account the fact that the Guerra-Talagrand bound is a two-dimensional extension of Guerra's original replica symmetry breaking bound. 
 	For $\gamma\in \mathcal{N}_v $, consider the following PDE solution,
 	\begin{align}\label{pde}
 	\partial_s\Psi_\gamma&=-\frac{\beta^2\xi''}{2}\bigl(\partial_{xx}\Psi_\gamma+\gamma\bigl(\partial_x\Psi_\gamma\bigr)^2\bigr)
 	\end{align}
 	for $(s,x)\in [0,v)\times \mathbb{R}$
 	with boundary condition 
 	\begin{align*}
 	\Psi_\gamma(v,x)&=\log\cosh x +\frac{\beta^2}{2}\bigl(\xi'(1)-\xi'(v)\bigr).
 	\end{align*}
 	Define a functional on $\mathcal{N}_v $ by
 	\begin{align*}
 	\mathcal{L}_v (\gamma)=2\Psi_\gamma(0,0)-\beta^2\Bigl(\int_0^v\gamma(s)\xi''(s)sds+\int_v^1\xi''(s)sds\Bigr).
 	\end{align*}
 	Note that if $\alpha\in \mathcal{M}$ with $\alpha(v)=1,$ then $\mathcal{L}_v (\gamma)=2\mathcal{P}_\beta(\alpha)$ for all $\gamma\in \mathcal{N}_v $ satisfying $\gamma(s)=\alpha(s)$ on $[0,v].$
 	The Guerra-Talagrand inequality states that
 	\begin{align}\label{rsb}
 	\e CF_{N,v}(\beta)&\leq \mathcal{L}_v (\gamma)
 	\end{align}
 	for all $\gamma\in \mathcal{N}_v .$ Take $\gamma_0\in\mathcal{N}_v $ by letting $\gamma_0(s)=\alpha_P(s)$ on $[0,v]$. Similar to the directional derivative of $\mathcal{P}_\beta$, we can also compute that for any $\gamma\in \mathcal{N}_u ,$
 	\begin{align*}
 	D_+\mathcal{L}_v ((1-\lambda)\gamma_0+\lambda \gamma)\Big|_{\lambda= 0} &=\beta^2\int_0^v\xi''(s)(\rho_\beta(s)-s)(\gamma(s)-\gamma_0(s))ds.
 	\end{align*}
 	Note that $\alpha_P(s)\equiv 1$ on $[0,1]$. If we fix $\gamma=2\alpha_P\in \mathcal{N}_v ,$ then from \eqref{extra:thm1:eq1}, the above equation implies
 	\begin{align*}
 	D_+\mathcal{L}_v ((1-\lambda)\gamma_0+\lambda \gamma)\Big|_{\lambda= 0} &=\beta^2\int_0^v\xi''(s)(\rho_\beta(s)-s)ds<0.
 	\end{align*}
 	This means that there exists some $\lambda_0>0$ such that
 	\begin{align*}
 \mathcal{L}_v((1+\lambda_0)\alpha_P)=\mathcal{L}_v ((1-\lambda_0)\gamma_0+\lambda_0 \gamma)<\mathcal{L}_v (\gamma_0)=2\mathcal{P}_\beta(\alpha_P).
 	\end{align*}
 	Note that the functional $\mathcal{L}_v$ and the space $\mathcal{N}_v$ can be defined for all $v\in [0,1].$ It also can  be checked (see, e.g.,  \cite{C14}) that $(v,\beta)\in[0,1]\times[a,b]\mapsto\inf_{0\leq \lambda\leq 1}\mathcal{L}_v((1+\lambda)\alpha_P)$ is continuous. Therefore, there exists a constant $\delta>0$ such that
 	\begin{align}\label{ineq2}
 	\inf_{0\leq \lambda\leq 1}\mathcal{L}_v((1+\lambda)\alpha_P)<2\mathcal{P}_\beta(\alpha_P)-\delta
 	\end{align}
 	for all $v\in [\varepsilon,1]\cap B_N$ and $\beta\in[a,b].$ Consequently, from \eqref{rsb} and this inequality, 
 	$$\e CF_{N,v}(\beta)\leq 2\mathcal{P}_\beta(\alpha_P)-\delta$$ for all $v\in [\varepsilon,1]\cap B_N$, $\beta\in[a,b],$ and $N\geq 1.$ Note that it is a classical result (see \cite{convex}) that if a sequence of convex functions is convergent on an open interval of $\mathbb{R}$, then it must be uniformly convergent on any compact subset of this open interval. Thus, $\e F_N(\beta)$ converges to $\mathcal{P}_\beta(\alpha_P)$ uniformly on $[a,b].$ From this, a standard application of the sub-Gaussian concentration inequality for the free energies (see, e.g., \cite[Appendix]{Chen14}) implies that there exists some $K>0$ independent of $N$ and $\delta'>0$ such that for any $\beta\in[a,b]$ and $v\in [\varepsilon,1]\cap B_N$, with probability at least $1-Ke^{-N/K}$,
 	\begin{align*}
 	CF_{N,v}(\beta)&\leq 2F_N(\beta)-\delta'
 	\end{align*}
 	for any $N\geq 1.$ 
 	This clearly implies that there exists a constant $K'>0$ such that for any $\beta\in[a,b]$, $$\e \bigl\la I(R_{1,2}\geq \varepsilon) \bigr\ra_\beta\leq KN\exp(-\delta'N)+Ke^{-N/K}\leq K'e^{-N/K'}$$ for all $N\geq 1.$ Finally, since $p$ is even, $\e \bigl\la I(R_{1,2}\geq \varepsilon) \bigr\ra_\beta=\e \bigl\la I(R_{1,2}\leq -\varepsilon) \bigr\ra_\beta.$ Thus, \eqref{extra:thm2:eq1} holds.
 \end{proof}
 
 Next we use the cavity method to establish Theorem \ref{extra:thm3} for even $p\geq 4$. Our proof is based on a mixed argument of \cite[Chapter 1]{Talbook1} and \cite[Chapter 13]{Talbook2}. Before we start the main proof, we define some notation and establish some auxiliary lemmas. 
  For any subset $A$ of $\{1,\ldots,p\}$, let $I_A$ be the collection of all $(i_1,\ldots,i_p)\in \{1,\ldots,N\}^p$ such that $i_a=N$ for all $a\in A$ and $i_a\in \{1,\ldots,N-1\}$ for all $a\notin A$. Set the Hamiltonian 
  \begin{align*}
  H_N^A(\sigma)&=\frac{1}{N^{(p-1)/2}}\sum_{(i_1,\ldots,i_p)\in I_A}g_{i_1,\ldots,i_p}\sigma_{i_1}\cdots\sigma_{i_p}
  \end{align*}
  for $\sigma\in\Sigma_N,$ where 
  \begin{align}\label{extra:cov}
  \e  H_N^A(\sigma^1)H_N^A(\sigma^2)= \frac{1}{N^{|A|-1}}\bigl(R_{1,2}^-\bigr)^{p-|A|}\bigl(\sigma_N^1\sigma_N^2\bigr)^{|A|}.
  \end{align}
  From this, we can express
  \begin{align*}
  H_N(\sigma)&=\sum_{j=0}^p\sum_{A:|A|=j}H_N^A(\sigma),
  \end{align*}
  Consider the interpolating Hamiltonian by
  \begin{align*}
  H_{N,t}(\sigma)&=\Bigl(\frac{N-1}{N}\Bigr)^{(p-1)/2}H_{N-1}(\sigma)+\sqrt{t}\sum_{j=1}^p\sum_{A:|A|=j}H_N^A(\sigma).
  \end{align*}
  Denote by $(\sigma^\ell)_{\ell\geq 1}$ a sequence of i.i.d. samplings from the Gibbs measure associated to the Boltzmann weight $\exp \bigl(\beta H_{N,t}/\sqrt{2}\bigr).$ For clarity, we omit the variable $\beta$ and denote the Gibbs average with respect to this sequence simply by $\la \cdot\ra_t$. Assume that $f$ is a bounded function of the i.i.d. $\sigma^1,\ldots,\sigma^n$ samplings. For $t\in[0,1]$, set $\nu_t(f)=\e \la f\ra_t.$ It is easy to check that $H_{N,t}=H_N$ when $t=1$. For simplicity we denote $\nu(f)=\nu_1(f).$ Set $R_{\ell,\ell'}^-=N^{-1}\sum_{i=1}^{N-1}\sigma_i^\ell\sigma_i^{\ell'}$.

\begin{lemma}\label{extra:lem1}
Let $f$ be a bounded function depending only on the i.i.d. samplings $\sigma^1,\ldots,\sigma^n.$ Then
\begin{align*}
\nu_t'(f)&=\frac{\beta^2}{2}\sum_{j=1}^p{p\choose j}\frac{1}{N^{j-1}}\Bigl[\sum_{\ell<\ell'}^{n}\nu_t\bigl( f \bigl(R_{\ell,\ell'}^-\bigr)^{p-j}\bigl(\sigma_N^\ell\sigma_N^{\ell'}\bigr)^{j}\bigr)\\
&\qquad\qquad\qquad\qquad-n\sum_{\ell=1}^{n}\nu_t\bigl( f \bigl(R_{\ell,n+1}^-\bigr)^{p-j}\bigl(\sigma_N^\ell\sigma_N^{n+1}\bigr)^{j}\bigr)\\
&\qquad\qquad\qquad\qquad+\frac{n(n+1)}{2}\nu_t\bigl(f \bigl(R_{n+1,n+2}^-\bigr)^{p-j}\bigl(\sigma_N^{n+1}\sigma_N^{n+2}\bigr)^{j}\bigr)\Bigr].
\end{align*}
\end{lemma}

\begin{proof}
Compute directly
\begin{align*}
\nu_t'(f)&=\frac{\beta}{2\sqrt{2t}}\sum_{j=1}^p\sum_{A:|A|=j}\e\Bigl\la f\Bigl(\sum_{\ell=1}^{n}H_N^A(\sigma^\ell)-nH_N^A(\sigma^{n+1})\Bigr)\Bigr\ra_t.
\end{align*}
Using Gaussian integration by parts and \eqref{extra:cov} gives
\begin{align*}
\e\Bigl\la f\sum_{\ell=1}^{n}H_N^A(\sigma^\ell)\Bigr\ra_t&=\frac{\beta\sqrt{t}}{N^{|A|-1}\sqrt{2}}\sum_{\ell,\ell'=1}^{n}\e\Bigl\la f \bigl(R_{\ell,\ell'}^-\bigr)^{p-|A|}\bigl(\sigma_N^\ell\sigma_N^{\ell'}\bigr)^{|A|}\Bigr\ra_t\\
&-\frac{n\beta\sqrt{t}}{N^{|A|-1}\sqrt{2}}\sum_{\ell=1}^{n}\e\Bigl\la f \bigl(R_{\ell,n+1}^-\bigr)^{p-|A|}\bigl(\sigma_N^\ell\sigma_N^{n+1}\bigr)^{|A|}\Bigr\ra_t
\end{align*}
and
\begin{align*}
\e\bigl\la fH_N^A(\sigma^{n+1})\Bigr\ra_t&=\frac{\beta\sqrt{t}}{N^{|A|-1}\sqrt{2}}\sum_{\ell=1}^{n+1}\e\Bigl\la f \bigl(R_{\ell,n+1}^-\bigr)^{p-|A|}\bigl(\sigma_N^\ell\sigma_N^{n+1}\bigr)^{|A|}\Bigr\ra_t\\
&-\frac{(n+1)\beta\sqrt{t}}{N^{|A|-1}\sqrt{2}}\e\Bigl\la f \bigl(R_{n+1,n+2}^-\bigr)^{p-|A|}\bigl(\sigma_N^{n+1}\sigma_N^{n+2}\bigr)^{|A|}\Bigr\ra_t.
\end{align*}
Plugging these two equations into $\nu_t'(f)$ and performing some rearrangement gives the announced result .
\end{proof}

\begin{lemma}\label{extra:lem2}
	Let $f$ be a nonnegative bounded function depending only on the i.i.d. samplings $\sigma^1,\ldots,\sigma^n.$ We have
\begin{align*}
\nu_t(f)&\leq \exp\bigl (2^{p}n^2\beta^2\bigr)\nu(f).
\end{align*}
\end{lemma}

\begin{proof}
Observe that from Lemma \ref{extra:lem1},
\begin{align*}
|\nu_t'(f)|&\leq 2^{p}n^2\beta^2 \nu_t(f).
\end{align*}
Thus,
\begin{align*}
-\nu_t'(f)\leq 2^{p}n^2\beta^2\nu_t(f).
\end{align*}
Taking integral on the interval $[t,1]$ leads to
\begin{align*}
\log \nu_t(f)-\log \nu_1(f)=-\int_t^1\frac{\nu_s'(f)}{\nu_s(f)}ds\leq 2^{p}n^2\beta^2(1-t)\leq 2^{p}n^2\beta^2,
\end{align*}
which finishes the proof of the claim.
\end{proof}

\begin{lemma}
	\label{extra:lem3}
	Assume that 
	\begin{align*}
	\nu\bigl(R_{1,2}^{2j}\bigr)\leq \frac{K_0}{N^j},\,\,\forall 0\leq j\leq k
	\end{align*}
	for some $K_0\geq 1.$ Then
	\begin{align*}
	\nu\bigl(|R_{1,2}|^j\bigr)&\leq \frac{K_0}{N^{j/2}},\,\,\forall 0\leq j\leq 2k,\\
	\nu\bigl(|R_{1,2}^-|^{2k}\bigr)&\leq \frac{2^{2k}K_0}{N^k}.
	\end{align*}
\end{lemma}

\begin{proof}
	The first inequality follows directly by H\"{o}lder's inequality and noting that $K_0\geq 1$. As for the second one, it can be established by noting that $R_{1,2}=R_{1,2}^-+\sigma_N^1\sigma_N^2/N$ and using the Binomial formula. Indeed,
	\begin{align*}
	\nu\bigl(|R_{1,2}^-|^{2k}\bigr)&\leq \sum_{j=0}^{2k}{2k\choose j}\frac{1}{N^{2k-j}}\nu\bigl(|R_{1,2}|^j\bigr)\\
	&\leq K_0\sum_{j=0}^{2k}{2k\choose j}\frac{1}{N^{2k-j}}\cdot \frac{1}{N^{j/2}}\\
	&\leq K_0\Bigl(\frac{1}{N}+\frac{1}{\sqrt{N}}\Bigr)^{2k}\\
	&\leq \frac{2^{2k}K_0}{N^k}.
	\end{align*}
\end{proof}

\begin{lemma}
	\label{extra:lem4}
	For any positive integer $l$, we have
	\begin{align*}
	\bigl|(R_{1,2}^-)^{l+1}-R_{1,2}^{l+1}\bigr|\leq \frac{l}{N}\bigl(|R_{1,2}^-|^{l}+|R_{1,2}|^l\bigr).
	\end{align*}
\end{lemma}

\begin{proof}
	The proof follows simply by using $|x^{l+1}-y^{l+1}|\leq l|x-y|(|x|^l+|y|^l)$ for any $x,y\in \mathbb{R}$.
\end{proof}

\begin{proof}[\bf Proof of Theorem \ref{extra:thm3}: even $p\geq 4$] 
Let $0<b<\beta_p $ be fixed. We proceed by induction. Assume that there exists some $K_0\geq 1$ independent of $N\geq 1$ such that for any $\beta\in [0,b],$
\begin{align*}
\nu\bigl(R_{1,2}^{2j}\bigr)\leq \frac{K_0}{N^j},\,\,\forall N\geq 1
\end{align*}
for all $0\leq j\leq k.$ Our goal is to prove that there exists a constant $K_0'\geq 1$ such that for any $\beta\in[0,b],$
\begin{align*}
\nu\bigl( R_{1,2}^{2(k+1)}\bigr)\leq \frac{K_0'}{N^{k+1}},\,\,\forall N\geq 1.
\end{align*}
To this end, write by symmetry among sites,
\begin{align*}
\nu(R_{1,2}^{2k+2})&=\nu\bigl(\sigma_N^1\sigma_N^2R_{1,2}^{2k+1}\bigr).
\end{align*}
Using Lemma \ref{extra:lem4} gives
\begin{align*}
\nu\bigl(\sigma_N^1\sigma_N^2R_{1,2}^{2k+1}\bigr)&=\nu\bigl(\sigma_N^1\sigma_N^2\bigl(R_{1,2}^-\bigr)^{2k+1}\bigr)+\mathcal{E},
\end{align*}
where from the induction hypothesis and Lemma \ref{extra:lem3},
\begin{align*}
|\mathcal{E}|&\leq \frac{(2k+1)}{N}\bigl(\nu\bigl(R_{1,2}^{2k}\bigr)+\nu\bigl((R_{1,2}^-)^{2k}\bigr)\bigr)\leq \frac{2(2k+1)C_1}{N^{k+1}},
\end{align*}
where
$
C_1:=2^{2k}K_0.
$
Set $f=\sigma_N^1\sigma_N^2(R_{1,2}^-)^{2k+1}.$ We control $\nu(f)$ by using the mean value theorem,
\begin{align*}
%\label{extra:eq9}
|\nu(f)-\nu_0(f)|\leq \sup_{t\in[0,1]}|\nu_t'(f)|.
\end{align*}
Note that under the measure $\nu_0$, $\sigma_N^1,\ldots,\sigma_N^{4}$ are i.i.d. Bernoulli random variables on $\{-1,1\}$ with mean zero and are independent of $R_{\ell,\ell'}^-$. This implies that $\nu_0(f)=0.$ In addition, from Lemma \ref{extra:lem1},
\begin{align*}
\nu_t'(f)&=\frac{\beta^2}{2}\sum_{j=1}^p{p\choose j}\frac{1}{N^{j-1}}\Bigl[\nu_t\bigl(  \bigl(R_{1,2}^-\bigr)^{2k+1+p-j}\bigl(\sigma_N^1\sigma_N^{2}\bigr)^{j+1}\bigr)\\
&\qquad\qquad\qquad\qquad-2\sum_{\ell=1,2}\nu_t\bigl(\bigl(R_{1,2}^-\bigr)^{2k+1} \bigl(R_{\ell,3}^-\bigr)^{p-j}\bigl(\sigma_N^1\sigma_N^2\bigr)\bigl(\sigma_N^\ell\sigma_N^{3}\bigr)^{j}\bigr)\\
&\qquad\qquad\qquad\qquad+3\nu_t\bigl(\bigl(R_{1,2}^-\bigr)^{2k+1}  \bigl(R_{3,4}^-\bigr)^{p-j}\bigl(\sigma_N^1\sigma_N^2\bigr)\bigl(\sigma_N^{3}\sigma_N^{4}\bigr)^{j}\bigr)\Bigr].
\end{align*}
It follows by H\"{o}lder's inequality and Lemma \ref{extra:lem2} that
\begin{align}
\begin{split}\notag
&|\nu_t'(f)|\\
&\leq \frac{\beta^2}{2}\sum_{j=1}^p{p\choose j}\frac{1}{N^{j-1}}\Bigl[\nu_t\bigl( |R_{1,2}^-|^{2k+1+p-j}\bigr)
+ 4\nu_t\bigl(|R_{1,2}^-|^{2k+1}|R_{1,3}^-|^{p-j}\bigr)+3\nu_t\bigl(|R_{1,2}^-|^{2k+1}  |R_{3,4}^-|^{p-j}\bigr)\Bigr]\\
&\leq 4\beta^2\sum_{j=1}^p{p\choose j}\frac{1}{N^{j-1}}\nu_t\bigl( |R_{1,2}^-|^{2k+1+p-j}\bigr)
\end{split}\\
\begin{split}\label{extra:eq2}
&\leq 4\beta^2\exp\bigl (2^{p+2}\beta^2\bigr)\sum_{j=1}^p{p\choose j}\frac{1}{N^{j-1}}\nu\bigl( |R_{1,2}^-|^{2k+1+p-j}\bigr).
\end{split}
\end{align}
Here, write
\begin{align*}
\sum_{j=1}^p{p\choose j}\frac{1}{N^{j-1}}\nu\bigl( |R_{1,2}^-|^{2k+1+p-j}\bigr)=p\nu\bigl( |R_{1,2}^-|^{2k+p}\bigr)+\sum_{j=2}^p{p\choose j}\frac{1}{N^{j-1}}\nu\bigl( |R_{1,2}^-|^{2k+1+p-j}\bigr).
\end{align*}
To control the first term, we use Lemma \ref{extra:lem4} to get
\begin{align*}
\Bigl|\nu\bigl( |R_{1,2}^-|^{2k+p}\bigr)-\nu\bigl( |R_{1,2}|^{2k+p}\bigr)\Bigr|&\leq \frac{(2k+p)}{N}\bigl(\nu\bigl( |R_{1,2}^-|^{2k+p-1}\bigr)+\nu\bigl( |R_{1,2}|^{2k+p-1}\bigr)\bigr)\\
&\leq \frac{(2k+p)}{N}\bigl(\nu\bigl( |R_{1,2}^-|^{2k}\bigr)+\nu\bigl( |R_{1,2}|^{2k}\bigr)\bigr)\\
&\leq \frac{(2k+p)C_1}{N^{k+1}}.
\end{align*}
As for the second term, 
\begin{align*}
\sum_{j=2}^p{p\choose j}\frac{1}{N^{j-1}}\nu\bigl( |R_{1,2}^-|^{2k+1+p-j}\bigr)&\leq \frac{2^p}{N}\nu\bigl(|R_{1,2}^-|^{2k+1}\bigr)\leq \frac{2^p}{N}\nu\bigl(|R_{1,2}^-|^{2k}\bigr)\leq \frac{2^pC_1}{N^{k+1}}.
\end{align*}
From \eqref{extra:eq2}, the above two inequalities give
\begin{align*}
|\nu_t'(f)|&\leq C_2(\beta)\nu\bigl( |R_{1,2}|^{2k+p}\bigr)+\frac{C_3(\beta)}{N^{k+1}}.
\end{align*}
for
\begin{align*}
C_2(\beta)&:=4p\beta^2\exp\bigl (2^{p+2}\beta^2\bigr),\\
C_3(\beta)&:=4\beta^2\exp\bigl (2^{p+2}\beta^2\bigr)\bigl((2k+p)C_1p+2^pC_1\bigr).
\end{align*}
As a result,
\begin{align}
\begin{split}\notag
\nu(R_{1,2}^{2k+2})&=\nu(f)\\
&\leq \nu_0(f)+\sup_{t\in[0,1]}|\nu_t'(f)|\\
&\leq C_2(\beta)\nu\bigl( |R_{1,2}|^{2k+p}\bigr)+\frac{C_3(\beta)}{N^{k+1}}
\end{split}\\
\begin{split}
\label{extra:eq12}
&\leq C_2(\beta)\nu\bigl( |R_{1,2}|^{2k+3}\bigr)+\frac{C_3(\beta)}{N^{k+1}},
\end{split}
\end{align}
where the last inequality used the assumption that $p\geq 3.$
Note that $C_2$ and $C_3$ are nondecreasing functions of $\beta.$
Now we divide our discussion into two cases:
\smallskip

{\noindent\bf Case I:} Let $0<a\leq b$ be the largest number such that $C_2(a)\leq \frac{1}{2}$. From the above inequality, since
\begin{align*}
\nu(R_{1,2}^{2k+2})&\leq C_2(a)\nu\bigl( |R_{1,2}|^{2k+2}\bigr)+\frac{C_3(a)}{N^{k+1}},
\end{align*}
we see that for any $\beta\in[0,a]$,
\begin{align*}
\nu(R_{1,2}^{2k+2})\leq \frac{2C_3(a)}{N^{k+1}}.
\end{align*}

\smallskip

{\noindent \bf Case II:} Let $a$ be the constant we chose from Case I. Fix $\varepsilon>0$ such that $\varepsilon C_2(b)<1/2$. For this $\varepsilon,$ we use Theorem \ref{extra:thm2} to find a constant $K>0$ independent of $N$ such that for any $\beta\in[a,b]$,
\begin{align*}
%\label{extra:eq10}
\nu\bigl(I\bigl(|R_{1,2}|\geq \varepsilon\bigr)\bigr)\leq K\exp\Bigl(-\frac{N}{K}\Bigr),\,\,\forall N\geq 1.
\end{align*} 
Therefore,
\begin{align*}
\nu\bigl( |R_{1,2}|^{2k+3}\bigr)&\leq \nu\bigl( I\bigl(|R_{1,2}|\geq \varepsilon\bigr)\bigr)+\nu\bigl( |R_{1,2}|^{2k+3}I\bigl(|R_{1,2}|<\varepsilon\bigr)\bigr)\\
&\leq K\exp\Bigl(-\frac{N}{K}\Bigr)+\varepsilon\nu\bigl( R_{1,2}^{2k+2}\bigr).
\end{align*}
From \eqref{extra:eq12}, this implies that
\begin{align*}
\nu(R_{1,2}^{2k+2})&\leq C_2(b)K\exp\Bigl(-\frac{N}{K}\Bigr)+\varepsilon C_2(b)\nu\bigl( R_{1,2}^{2k+2}\bigr)+\frac{C_3(b)}{N^{k+1}}.
\end{align*}
Since $\varepsilon C_2(b)<1/2,$ this inequality deduces that for any $\beta\in[a,b],$
\begin{align*}
\nu(R_{1,2}^{2k+2})\leq 2\Bigl(C_2(b)K\exp\Bigl(-\frac{N}{K}\Bigr)+\frac{C_3(b)}{N^{k+1}}\Bigr).
\end{align*}
Combining the above two cases together implies the existence of a constant $K_0'$ independent of $N$ such that 
\begin{align*}
\nu(R_{1,2}^{2k+2})\leq \frac{K_0'}{N^{k+1}}
\end{align*}
for all $\beta\in[0,b]$ and $N\geq 1$.
This completes our proof.
\end{proof}

\begin{remark}\rm
We have seen in Case II that if $p\geq 3$, it is possible to obtain moment control for the overlap for big $\beta$ by using \eqref{extra:eq12} as long as we know the tail control of the overlap.	If $p=2$, \eqref{extra:eq12} loses its power and one must go further to consider the second derivative of $\nu(R_{1,2}^{2k+2})$ in order to obtain the moment control in the Sherrington-Kirkpatrick model. See \cite[Chapter 11]{Talbook2} for detail.
\end{remark}

\subsubsection{Odd $p$ case}

In this subsection, we establish Theorems \ref{extra:thm2} and \ref{extra:thm3} for odd $p\geq 3.$ It relies on the validity of the Guerra-Talagrand bound stated in the following proposition. Recall the coupled free energy $CF_{N,v}(\beta)$ from \eqref{ineq1}. For any $v\in [0,1]$ and $\varepsilon>0,$ define
\begin{align*}
CF_{N,v,\varepsilon}(\beta)&=\frac{1}{N}\log \sum\exp \frac{\beta}{\sqrt{2}} \bigl(H_N(\sigma^1)+H_N(\sigma^2)\bigr), 
\end{align*}
where the sum is over all $\sigma^1,\sigma^2\in \Sigma_N$ with $R_{1,2}\in [v-\varepsilon,v+\varepsilon].$ Recall  $\mathcal{L}_v(\gamma)$ from \eqref{rsb}.

\begin{proposition}\label{odd:lem2}
	Let $p\geq 3$ be an odd number. For any $v\in [0,1]$, we have
	\begin{align*}
	%\label{odd:lem2:eq1}
	\lim_{\varepsilon\downarrow 0}\limsup_{N\rightarrow\infty}\e CF_{N,v,\varepsilon}(\beta)&\leq \mathcal{L}_v(\gamma)
	\end{align*}	
	if $\gamma\in \mathcal{N}_v$ is a constant function on $[0,v].$ 
\end{proposition}

As we mentioned before, for even $p\geq 2$ the Guerra-Talagrand bound \eqref{rsb} is valid for all $\gamma\in \mathcal{N}_v$, but it is not clear if the same inequality holds for odd $p.$ Nevertheless, this proposition states that the Guerra-Talagrand bound is true in the limit $N\rightarrow\infty$ as long as we consider $v\in [0,1]$ and take only constant $\gamma$'s. This choice of $\gamma$ is called the one-step replica symmetry breaking functional order parameter in spin glass literature. We now turn to the proof of Theorems  \ref{extra:thm2} and \ref{extra:thm3}. We defer the proof of Proposition \ref{odd:lem2} to Subsection \ref{final}.

\begin{proof}[\bf Proof of Theorem \ref{extra:thm2}: odd $p\geq 3$] Let $0<a<b<\beta_p $.
	Let $\varepsilon_0>0$ be fixed. Recall that in the argument for Theorem \ref{extra:thm2} with even $p$, we have that from \eqref{ineq2}, there exists some  $\delta>0$ such that
	\begin{align*}
	\inf_{0\leq \lambda\leq 1}\mathcal{L}_v((1+\lambda)\alpha_P)<2\mathcal{P}_\beta(\alpha)-\delta
	\end{align*}
	for all $v\in [\varepsilon_0,1]$ and $\beta\in [a,b].$ As this part of the argument does not rely on whether $p$ is odd or even, the same inequality is still valid for odd $p$. Since $(1+\lambda)\alpha_P\equiv 1+\lambda$ on $[0,v]$, it follows from Proposition \ref{odd:lem2} that
	\begin{align*}
	\lim_{\varepsilon\downarrow 0}\limsup_{N\rightarrow\infty}\e CF_{N,v,\varepsilon}(\beta)\leq 2\mathcal{P}_\beta(\alpha_P)-\delta
	\end{align*}
	for any $v\in [\varepsilon_0,1]$ and $\beta\in [a,b].$ Since the left-hand side of this inequality is in the logarithmic scale, from the sub-Gaussian concentration inequality for the free energies \cite[Proposition 9]{Chen14}, it can be argued by a covering argument that for any $\beta\in[a,b]$,
	\begin{align*}
	\limsup_{N\rightarrow\infty}\frac{1}{N}\e\log \sum_{R_{1,2}\in [\varepsilon_0,1]}\frac{1}{2^{2N}}\exp \frac{\beta}{\sqrt{2}}\bigl(H_N(\sigma^1)+H_N(\sigma^2)\bigr)\leq 2\mathcal{P}_\beta(\alpha_P)-\frac{\delta}{2}.
	\end{align*}
	Now using the uniform convergence of the free energies for $\beta\in [a,b]$ on both sides and the Parisi formula, there exists a constant $\delta'>0$ such that
	\begin{align*}
	\frac{1}{N}\e\log \sum_{R_{1,2}\in [\varepsilon_0,1]}\frac{1}{2^{2N}}\exp \frac{\beta}{\sqrt{2}}\bigl(H_N(\sigma^1)+H_N(\sigma^2)\bigr)\leq 2\e F_N(\beta)-\delta'
	\end{align*}
	for all $N\geq 1$ and $\beta\in[a,b].$ From this, we can argue by another application of the sub-Gaussian concentration inequality for the free energies in an identical manner as the last part of the proof of Theorem \ref{extra:thm2} with even $p$ to get that there exists some positive constant $K$ such that
	\begin{align*}
	\e \bigl\la I(R_{1,2}\in [\varepsilon_0,1])\bigr\ra_\beta\leq Ke^{-N/K}
	\end{align*}
	for any $N\geq 1$ and $\beta\in[a,b],$ where $\la\cdot\ra_\beta$ is the Gibbs average with respect to the product of the Gibbs measure $G_{N,\beta}$ in \eqref{gibbs}. 
	It remains to show that there exists some $K'>0$ such that 
	\begin{align}\label{odd:eq5}
	\e \bigl\la I(R_{1,2}\in [-1,-\varepsilon_0])\bigr\ra_\beta\leq K'e^{-N/K'}
	\end{align}
	for any $N\geq 1$ and $\beta\in[a,b].$ To see this, from Jensen's inequality,
	\begin{align*}
	&\frac{1}{N}\e\log \sum_{R_{1,2}\in [-1,-\varepsilon_0]}\frac{1}{2^{2N}}\exp \frac{\beta}{\sqrt{2}}\bigl(H_N(\sigma^1)+H_N(\sigma^2)\bigr)\\
%	&\leq \frac{1}{N}\log\sum_{R_{1,2}\in[-1,-\varepsilon_0]}\frac{1}{2^{2N}}\e \exp \frac{\beta}{\sqrt{2}}\bigl(H_N(\sigma^1)+H_N(\sigma^2)\bigr)\\
	&\leq\frac{1}{N}\log\sum_{R_{1,2}\in[-1,-\varepsilon_0]}\frac{1}{2^{2N}}\exp \frac{\beta^2N}{4}\bigl(2+2R_{1,2}^p\bigr)\leq \frac{\beta^2}{4}\bigl(2-2\varepsilon_0^p\bigr),
	\end{align*}
	where the last inequality used the fact that $p$ is odd.
	Since $\lim_{N\rightarrow\infty}F_N(\beta)=\beta^2/4$ for $\beta\in[a,b],$ a similar argument as above gives \eqref{odd:eq5}.
\end{proof}

\begin{proof}[\bf Proof of Theorem \ref{extra:thm3}: odd $p\geq 3$]
	Note that given the validity of Theorem \ref{extra:thm2}, the assumption of $p$ being even is not important in the proof of Theorem \ref{extra:thm3} for even $p$, so the statement of Theorem \ref{extra:thm3} for odd $p$ remains valid by the same cavity argument and using Theorem \ref{extra:thm2} for odd $p.$
\end{proof}

\subsection{Proof of Proposition \ref{prop1}}

Recall that $R_{1,2}$ is the overlap between two i.i.d. samplings $\sigma^1$ and $\sigma^2$ from the Gibbs measure \eqref{gibbs} and $\la \cdot\ra_\beta$ is the Gibbs average with respect to these random variables.

\begin{lemma}
	\label{extra:lem9}
	For any $p\geq 2$ and $\beta\geq 0$, we have
	\begin{align*}
	\p\Bigl(\Bigl|F_N(\beta)-\frac{\beta^2}{4}\Bigr|\geq r\Bigr)&\leq \frac{2}{r^2N}\e \bigl\la R_{1,2}^p\bigr\ra_\beta+\frac{1}{2r^2}\Bigl(\int_0^\beta l\e \bigl\la R_{1,2}^p\bigr\ra_{l}dl\Bigr)^2,\quad\forall r>0.
	\end{align*}
\end{lemma}

\begin{proof}
	For any $r>0,$  Chebyshev's inequality yields
	\begin{align}
	\begin{split}
	\notag
	\p\Bigl(\Bigl|F_N(\beta)-\frac{\beta^2}{4}\Bigr|\geq r\bigr)&\leq \frac{1}{r^2}\e\Bigl(F_N(\beta )-\frac{\beta ^2}{4}\Bigr)^2
	\end{split}\\
	\begin{split}
	\label{extra:lem9:proof:eq2}
	&\leq \frac{2}{r^2}\Bigl(\mbox{Var}(F_N(\beta ))+\Bigl(\e F_N(\beta )-\frac{\beta ^2}{4}\Bigr)^2\Bigr),
	\end{split}
	\end{align}
	where the second inequality used the trivial bound $(a+b)^2\leq 2a^2+2b^2$. The two terms in the last inequality can be controlled as follows. First, from the Poincar\'e inequality for Gaussian measure, since $$\partial_{g_{i_1,\ldots,i_p}}F_N(\beta )=\frac{\beta}{N^{(p+1)/2}}\bigl\la \sigma_{i_1}\cdots\sigma_{i_p}\bigr\ra_\beta ,$$
	we have
	\begin{align}\label{extra:lem9:proof:eq1}
	\mbox{Var}(F_N(\beta ))&\leq \e\sum_{1\leq i_1,\ldots,i_p\leq N}\Bigl(\frac{\beta}{N^{(p+1)/2}}\bigl\la \sigma_{i_1}\cdots\sigma_{i_p}\bigr\ra_\beta \Bigr)^2=\frac{\beta^2}{N}\e\bigl\la R_{1,2}^p\bigr\ra_\beta .
	\end{align}
	On the other hand, using Gaussian integration by parts, 
	\begin{align*}
	\frac{d}{d\beta }\e F_N(\beta )&=\frac{1}{\sqrt{2}}\e \Bigl\la \frac{H_N(\sigma)}{N}\Bigr\ra_\beta =\frac{\beta }{2}\bigl(1-\e \bigl\la R_{1,2}^p\bigr\ra_\beta \bigr)
	\end{align*}
	and thus,
	\begin{align}
	\begin{split}\notag
	\e F_N(\beta )&=\e F_N(0)+\int_0^\beta  \frac{d}{dl}\e F_N(l)dl\\
	&=\int_0^\beta \frac{l}{2}\bigl(1-\e \bigl\la R_{1,2}^p\bigr\ra_l\bigr)dl
	\end{split}\\
	\begin{split}
	\label{extra:lem9:proof:eq3}
	&=\frac{\beta ^2}{4}-\frac{1}{2}\int_0^\beta  l\e \bigl\la R_{1,2}^p\bigr\ra_ldl,
	\end{split}
	\end{align}
	which implies 
	\begin{align*}
	\Bigl(\e F_N(\beta )-\frac{\beta ^2}{4}\Bigr)^2&=\frac{1}{4}\Bigl(\int_0^\beta  l\e \bigl\la R_{1,2}^p\bigr\ra_ldl\Bigr)^2.
	\end{align*}
	This together with \eqref{extra:lem9:proof:eq2} and \eqref{extra:lem9:proof:eq1} clearly completes our proof.
\end{proof}

\begin{proof}[\bf Proof of Proposition \ref{prop1}] Let $0<b<\beta_p $ be fixed. From Theorem \ref{extra:thm3}, there exists a constant $K$ such that 
	\begin{align*}
	\e \bigl\la |R_{1,2}|^p\ra_\beta\leq \frac{K}{N^{p/2}}
	\end{align*}
	for any $N\geq 1$ and $0\leq \beta\leq b.$ Here we used Jensen's inequality in the case that $p$ is odd. From Lemma \ref{extra:lem9}, for any $r>0,$
	\begin{align*}
		\p\Bigl(\Bigl|F_N(\beta)-\frac{\beta^2}{4}\Bigr|\geq r\Bigr)&\leq \frac{2K}{r^{2}N^{p/2+1}}+\frac{1}{2r^{2}}\Bigl(\int_0^\beta \frac{lK}{N^{p/2}}dl\Bigr)^2\\
		&=\frac{2K}{r^{2}N^{p/2+1}}+\frac{\beta^4K^2}{8r^{2}N^{p}}\\
		&\leq \frac{1}{r^2N^{p/2+1}}\Bigl(2K+\frac{b^4K^2}{8}\Bigr).
	\end{align*}
	This completes our proof.
\end{proof}

\subsection{Proof of Proposition \ref{odd:lem2}}\label{final}

This subsection is devoted to establishing Proposition \ref{odd:lem2}. We briefly outline our argument as follows. Our approach is based on an interpolation of the coupled free energy $CF_{N,v,\varepsilon}(\beta)$ via Guerra's replica symmetry breaking scheme. In the case of the mixed even $p$-spin models, the same interpolation was also considered in Talagrand \cite{Talbook2}, where it was understood that the error term in the derivative of the interpolated coupled free energy consists of a sum of four main parts, in terms of the overlap matrix and the functional order parameter. In this case, it can be easily checked (see Remark \ref{rmk4}) that each of the four parts in the error term is nonnegative, which critically determines the validity of the Guerra-Talagrand inequality. In our case, the model involves purely odd $p$-spin interactions. Due to the oddness of $p$, it is not clear that the error term remains nonnegative. To overcome this difficulty, we  add a vanishing perturbation of higher order spin interactions to the coupled systems in the same manner as that performed in Panchenko \cite{Panchenko2015}. On the one hand, this will not influence the coupled free energy in the thermodynamic limit. On the other hand, this perturbation will be strong enough such that the system satisfies the Ghirlanda-Guerra identities and then leads to the ultrametricity of the overlaps \cite{PanUltra}. One of the important consequences of the ultrametricity, discovered by Panchenko \cite{PanMultipSP,Panchenko2015}, is that the entries in the overlap matrix will be synchronized in the sense that the whole overlap matrix becomes asymptotically symmetric and positive semi-definite. From these, if we in particular choose the functional order parameters to be of one-step replica symmetry breaking, the four main parts together in the error term exhibits a nonnegative sign and gives the desired upper bound with this choice of the order parameters.  

We now turn to our proof. Throughout the remainder of this subsection, let $p\geq 3$ be a fixed odd number. The following lemma is crucial to our argument. Define
\begin{align*}
\Gamma(v,x)&=x^p-pxv^{p-1}+(p-1)v^{p}
\end{align*}
for $v\in[0,1]$ and $x\in [-1,1].$

\begin{figure}[h]
	\centering
	\includegraphics[width=0.5\textwidth]{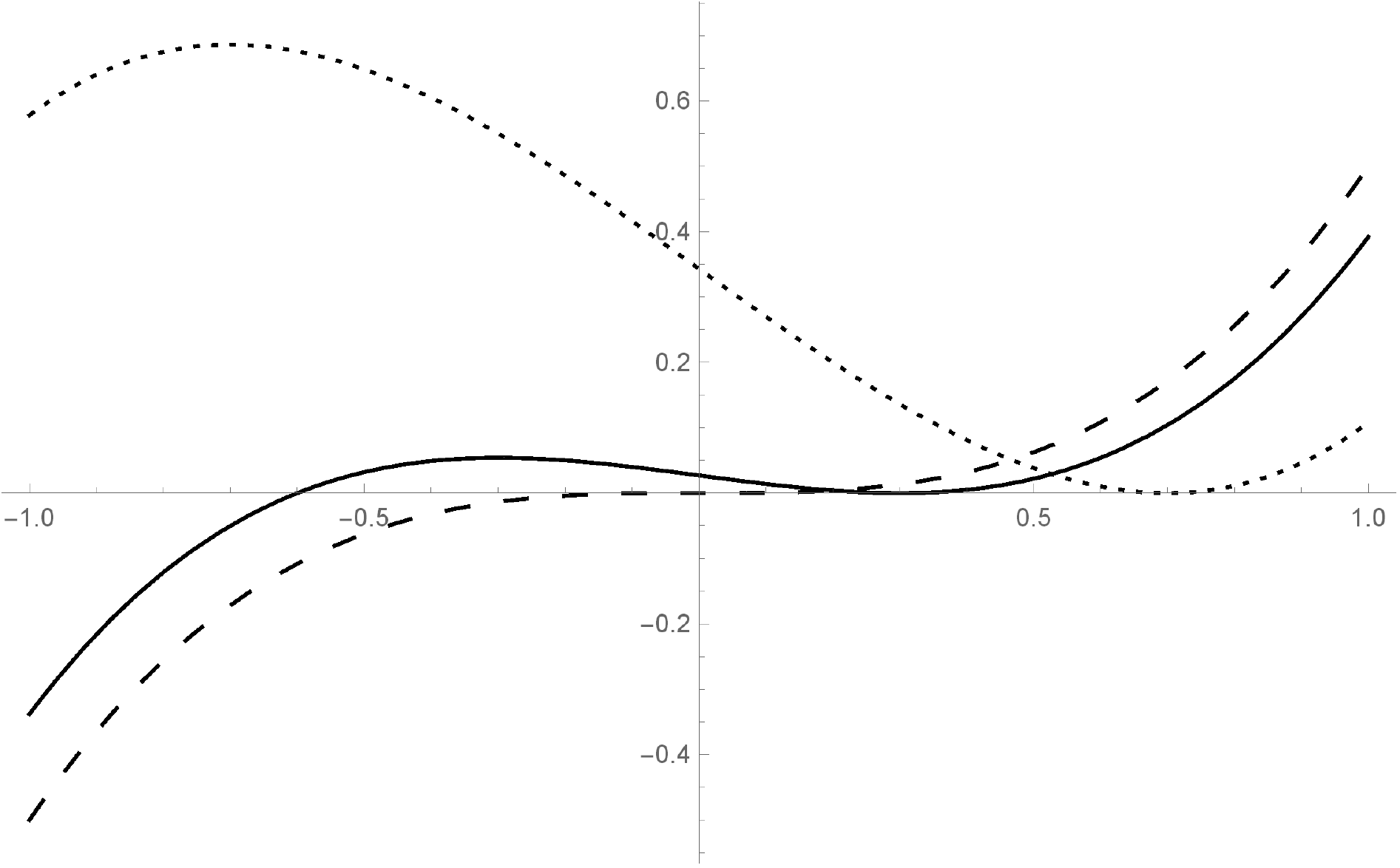}
	\caption{Graph for $\Gamma(v,\cdot)$ with $p=3$ and $v=0$ (dashed line), $=0.3$ (solid line), $=0.7$ (dotted line)}
	\label{fig:mesh1}
\end{figure}

\begin{lemma}\label{odd:lem1}
	For any $x,y,z\in [-1,1]$, if the matrix
	\begin{align}\label{odd:lem1:eq1}
	\left[
	\begin{array}{cc}
	x&z\\
	z&y
	\end{array}\right]
	\end{align}
	is positive semi-definite,
	then
	$
	\Gamma(v,x)+\Gamma(v,y)+2\Gamma(v,z)\geq 0
	$
	for any $v\in [0,1].$
\end{lemma}

\begin{proof} 
	Let $v\in[0,1]$ be fixed. For convenience, set
	\begin{align*}
	f(x)&=\Gamma(v,x).
	\end{align*} 
	Let us begin by stating some properties of $f$ (see Figure \ref{fig:mesh1}), which can be easily verified by a direct computation. First, since $p$ is odd, $f$ is convex on $[0,1]$ and concave on $[-1,0].$ In addition, $f\geq 0$ on $[0,1];$ on the interval $[-1,0],$ either $f\geq 0$ on $[-1,0]$ or there exists a unique $x_0\geq v$ such that $f\leq 0$ on $[-1,-x_0]$ and $f\geq 0$ on $[-x_0,0].$ Finally, $f'\geq 0$ on $[-1,-v].$

	Next, we turn to the proof of the announced inequality. Note that the positive semi-definiteness of \eqref{odd:lem1:eq1} ensures $x,y\geq 0$ and $|z|\leq \sqrt{xy}.$ From the above stated properties of $f$, it suffices to assume that $x_0$ exists and $z\leq-x_0.$ Since $f$ is convex on $[0,1]$ and $x,y\geq 0,$ 
	\begin{align*}
	f\Bigl(\frac{x+y}{2}\Bigr)\leq \frac{f(x)+f(y)}{2}.
	\end{align*}
	Also since $f'\geq 0$ on $[-1,-x_0]$ and $-(x+y)/2\leq -\sqrt{xy}\leq z$, we see that $$
	f(z)\geq f(-\sqrt{xy})\geq f\Bigl(-\frac{x+y}{2}\Bigr).$$
	From these,
	\begin{align*}
	f(x)+f(y)+2f(z)&\geq 2\Bigl(f\Bigl(\frac{x+y}{2}\Bigr)+f\Bigl(-\frac{x+y}{2}\Bigr)\Bigr).
	\end{align*}
	Our proof is then completed by noting that $f(c)+f(-c)\geq 0$ for any $c\in[0,1].$ This holds because $p$ is odd and this yields
	\begin{align*}
	f(c)+f(-c)&=c^p-pcv^{p-1}+(p-1)v^{p}+\bigl(-c^p+pcv^{p-1}+(p-1)v^{p}\bigr)=2(p-1)v^p\geq 0.
	\end{align*}
\end{proof}

To establish the proof of Proposition \ref{odd:lem2}, we first make some preparations in four steps.  

\smallskip

{\noindent \bf Step 1: Perturbation.} We construct a Hamiltonian of smaller order in the same fashion as that in \cite{Panchenko2015}. It will be used later as a perturbation along Guerra's interpolation. Let $(\lambda_l)_{l\geq 1}$ be a numeration of $\mathbb{Q}\cap[-1,1].$  For $\sigma\in \Sigma_N$ and $e=(i_1,\ldots,i_r)\in \{1,\ldots,N\}^r$,  denote
\begin{align*}
\sigma_{e}=\sigma_{i_1}\cdots\sigma_{i_r}.
\end{align*}
For any $\sigma^1,\sigma^2\in \Sigma_N$ and $l,{l'}\geq 1$, we set
\begin{align*}
S_{r,I}^{l,l'}(\sigma^1,\sigma^2)&=\bigl(\lambda_l\sigma_{e_1}^1+\lambda_{l'}\sigma_{e_1}^2\bigr)\cdots\bigl(\lambda_l\sigma_{e_n}^1+\lambda_{l'}\sigma_{e_n}^2\bigr)
\end{align*}
for any $I=(e_1,\ldots, e_n)\in \bigl(\{1,\ldots,N\}^r\bigr)^n.$ For $m,n_1,\ldots,n_m,l_1,l_1',\ldots,l_m,l_m'\geq 1,$ denote by 
\begin{align}\label{index}
\omega&=\bigl(m,n_1,\ldots,n_m,l_1,l_1',\ldots,l_m,l_m'\bigr).
\end{align}
Set 
\begin{align*}
X_{N}^\omega\bigl(\sigma^1,\sigma^2\bigr)&=\frac{1}{N^{r(n_1+\cdots+n_m)/2}}\sum_{I_1,\ldots,I_m}g_{r,I_1,\ldots,I_m}^{l_1,l_1',\ldots,l_m,l_m'}\prod_{j=1}^mS_{r,I_j}^{l_j,l_j'}(\sigma^1,\sigma^2),
\end{align*}
where $I_j=(e_1,\ldots,e_{n_j})\in \bigl(\{1,\ldots,N\}^r\bigr)^{n_j}$ for $j=1,\ldots,m$ and $g_{r,I_1,\ldots,I_m}^{l_1,l_1',\ldots,l_m,l_m'}$ are standard normal random variables and are independent of all indices and other randomness. For any $(\sigma^1(k),\sigma^2(k))$ and $(\sigma^1(k'),\sigma^2(k'))\in\Sigma_N^2$, set the overlap matrix $Q_{k,k'}$ by
\begin{align*}
Q_{k,k'}=\left[
\begin{array}{cc}
Q_{k,k'}^{1,1}&Q_{k,k'}^{1,2}\\
Q_{k,k'}^{2,1}&Q_{k,k'}^{2,2}
\end{array}\right]:=
\left[
\begin{array}{cc}
R\bigl(\sigma^1(k),\sigma^{1}(k')\bigr)&R\bigl(\sigma^1(k),\sigma^{2}(k')\bigr)\\
R\bigl(\sigma^2(k),\sigma^{1}(k')\bigr)&R\bigl(\sigma^2(k),\sigma^{2}(k')\bigr)
\end{array}\right],
\end{align*}
where $R(\sigma,\sigma'):=N^{-1}\sum_{i=1}^N\sigma_i\sigma_i'$ for any $\sigma,\sigma'\in \Sigma_N.$
The covariance of $X_{N}^\omega $ can be computed through 
\begin{align}\label{add:eq1}
\e X_N^\omega \bigl(\sigma^1(k),\sigma^2(k)\bigr)X_{N}^\omega (\sigma^1(k'),\sigma^2(k')\bigr)&=\prod_{j=1}^m\bigl\la Q_{k,k'}^{\circ r}\lambda_{l_j,l_j'},\lambda_{l_j,l_j'}\bigr\ra^{n_j},
\end{align}
where $Q_{k,k'}^{\circ r}$ is the Hadamard product of $Q_{k,k'}$ for $r$ times, 
$\lambda_{l,l'}$ is the column vector $(\lambda_{l},\lambda_{l'}),$ and $\la \cdot,\cdot\ra$ means the usual dot product. Note that \eqref{add:eq1} is bounded by $4^{n_1+\cdots+n_m}.$ Let $\theta=(\theta_\omega )_{\omega}$ be a sequence of i.i.d. random variables sampled uniformly at random from the interval $[1,2]$ indexed by $\omega$ defined in \eqref{index}. From these, set the perturbation Hamiltonian as
\begin{align*}
X_{N}^{\theta}\bigl(\sigma^1,\sigma^2\bigr)&=\sum_{\omega}\frac{\theta_\omega }{8^{|\omega|}}X_{N}^\omega \bigl(\sigma^1,\sigma^2\bigr),
\end{align*}
where $|\omega|$ is the sum of all entries in the index $\omega.$
Then conditioning on $\theta$, the covariance of $X_{N}^{\theta}\bigl(\sigma^1(k),\sigma^2(k)\bigr)$ and $X_{N}^{\theta}\bigl(\sigma^1(k'),\sigma^2(k')\bigr)$ is given by
\begin{align*}
\sum_{\omega}\frac{(\theta_{\omega})^2}{64^{|\omega|}}\prod_{j=1}^m\bigl\la Q_{k,k'}^{\circ r}\lambda_{l_j,l_j'},\lambda_{l_j,l_j'}\bigr\ra^{n_j}.
\end{align*}

\smallskip

{\noindent \bf Step 2: Guerra's Interpolation scheme.} In what follows, we introduce Guerra's interpolation scheme for the coupled free energy.  Fix $v\in [0,1]$. Let $m_0=0<m_1<m_2<m_3=1.$ Let $(d_j)_{j\in \mathbb{N}}$ be the non-decreasing rearrangement of a Poisson point process of intensity measure $x^{-m_1-1}dx$. For each $j_1$, let $(d_{j_1,j})_{j\in \mathbb{N}}$ be a non-decreasing rearrangement of a Poisson point process of intensity $x^{-m_2-1}dx.$ These are all independent of each other for all $j_1$ and of $(d_j).$ For each $\eta=(j_1,j_2)\in \mathbb{N}^2,$ we define $$d_\eta^*=d_{j_1}d_{j_1,j_2}.$$
The Ruelle probability cascades is defined by the probability weights $(c_\eta)_{\eta\in \mathbb{N}^2}$,
\begin{align*}
c_\eta=\frac{d_\eta^*}{\sum_{\eta'\in \mathbb{N}^2} d_{\eta'}^*}.
\end{align*}
Set
\begin{align*}
\rho_0^{1,1}=0,\,\,\rho_1^{1,1}=v,\,\,\rho_2^{1,1}=1,\\
\rho_0^{1,2}=0,\,\,\rho_1^{1,2}=v,\,\,\rho_2^{1,2}=v,\\
\rho_0^{2,1}=0,\,\,\rho_1^{2,1}=v,\,\,\rho_2^{2,1}=v,\\
\rho_0^{2,2}=0,\,\,\rho_1^{2,2}=v,\,\,\rho_2^{2,2}=1.
\end{align*}
Let $(z_1^1,z_1^2)$ and $(z_2^1,z_2^2)$ be two independent Gaussian random vectors with mean zero and covariance,
\begin{align*}
\e z_a^\ell  z_a^{\ell '}=\xi'(\rho_{a}^{\ell ,\ell '})-\xi'(\rho_{a-1}^{\ell ,\ell '})
\end{align*}
for $a=1,2$ and $\ell ,\ell '=1,2.$ Let $(z_{1,j_1}^1,z_{1,j_1}^2)_{j_1\in \mathbb{N}}$ and $(z_{i,1,j_1}^1,z_{i,1,j_1}^2)_{j_1\in \mathbb{N}}$ for $1\leq i\leq N$ be i.i.d. copies of $(z_1^1,z_1^2).$
Also, let $(z_{2,j_1,j_2}^1,z_{2,j_1,j_2}^2)_{(j_1,j_2)\in \mathbb{N}^2}$ and $(z_{i,2,j_1,j_2}^1,z_{i,2,j_1,j_2}^2)_{(j_1,j_2)\in \mathbb{N}^2}$ for $1\leq i\leq N$ be i.i.d. copies of $(z_2^1,z_2^2)$. We assume that these are all independent of each other. For $(\sigma^1,\sigma^2,\eta)\in \Sigma_N^2\times\mathbb{N}^2,$ set
\begin{align*}
H_{N}^0\big(\sigma^1,\sigma^2,\eta\big)&=\sum_{\ell =1,2}\sum_{i=1}^N\sigma_i^\ell \bigl(z_{i,1,j_1}^\ell +z_{i,2,j_1,j_2}^\ell \bigr).
\end{align*} 
Note that for any two $\bigl(\sigma^1(1),\sigma^2(1),\eta(1)\bigr),\bigl(\sigma^1(2),\sigma^2(2),\eta(2)\bigr)\in \Sigma_N^2\times \mathbb{N}^2,$ the covariance of this Hamiltonian can be computed as
\begin{align}\label{ham0}
\e H_{N}^0\big(\sigma^1(1),\sigma^2(1),\eta(1)\big)H_{N}^0\big(\sigma^1(2),\sigma^2(2),\eta(2)\bigr)=\sum_{\ell,\ell'=1,2}Q_{1,2}^{\ell,\ell'}\xi'\bigl(\rho_{\eta(1)\wedge \eta(2)}^{\ell,\ell'}\bigr),
\end{align}
where for $\eta(1)=(j_1(1),j_2(1))$ and $\eta(2)=(j_1(2), j_2(2)),$ 
\begin{align*}
\eta(1)\wedge \eta(2)&=\left\{
\begin{array}{ll}
0,&\mbox{if $j_1(1)\neq j_1(2)$ and $j_2(1)\neq j_2(2)$},\\
1,&\mbox{if $j_1(1)= j_1(2)$ and $j_2(1)\neq j_2(2)$},\\
2,&\mbox{if $j_1(1)= j_1(2)$ and $j_2(1)= j_2(2)$}.
\end{array}
\right.
\end{align*}

Let $1/4<\chi<1/2$ be a fixed constant. For $t\in[0,1]$, define the interpolating Hamiltonian with perturbation by
\begin{align*}
H_{N}^{t,\theta}\bigl(\sigma^1,\sigma^2,\eta\bigr)&=\frac{\sqrt{t}}{\sqrt{2}}\bigl(H_N\bigl(\sigma^1\bigr)+H_N\bigl(\sigma^2\bigr)\bigr)\\
&+\sqrt{1-t}H_{N}^0\bigl(\sigma^1,\sigma^2,\eta\bigr)+N^{\chi}X_{N}^{\theta}\bigl(\sigma^1,\sigma^2\bigr).
\end{align*}
Note that $H_N$ and $H_N^0$ are of order $\sqrt{N}$ and $X_{N,\theta}$ is of order $1$, so $N^{\chi}X_{N}^{\theta}$ should be regarded as a vanishing perturbation. 
Let $\varepsilon>0$ be fixed. Set
\begin{align*}
\Lambda_N(\varepsilon)&:=\bigl\{\bigl(\sigma^1,\sigma^2\bigr)\in \Sigma_N^2\times \mathbb{N}^2:R_{1,2}\in [v-\varepsilon,v+\varepsilon]\bigr\}.
\end{align*}
Define the interpolating free energy as
\begin{align*}
%\label{ham}
CF_{N}^{t,\theta}(\beta)&=\frac{1}{N}\log \sum_{\Lambda_N(\varepsilon)}\frac{c_\eta}{2^{2N}}\exp \beta H_{N}^{t,\theta}\bigl(\sigma^1,\sigma^2,\eta\bigr)
\end{align*}
and the Gibbs measure by $$G_{N}^{t,\theta}(\sigma^1,\sigma^2,\eta)=\frac{1}{Z_{N}^{t,\theta}}\cdot\frac{c_\eta}{2^{2N}}\exp\beta H_{N}^{t,\theta}(\sigma^1,\sigma^2,\eta)$$
for $(\sigma^1,\sigma^2,\eta)\in \Lambda_N(\varepsilon)$, where $Z_{N}^{t,\theta}$ is the normalizing constant. From now on, we denote $$
(\sigma^1(k),\sigma^2(k),\eta(k))_{k\geq 1}$$ as i.i.d. samplings from this Gibbs measure and the Gibbs average with respect to this sequence is denoted by $\la \cdot\ra_{N}^{t,\theta}.$ For $t\in[0,1]$, define
\begin{align*}
\psi_{N}(t)=\e_\theta\e CF_{N}^{t,\theta}(\beta),
\end{align*}
where the expectation $\e$ is with respect to the Gaussian randomness and the Ruelle probability cascades, while $\e_\theta$ is respect to the uniform random variable $\theta$ only. Now following the proof of Lemma 15.7.2 in \cite{Talbook2} using Gaussian integration by parts and \eqref{ham0} gives
\begin{align}
\begin{split}\label{derivative}
\psi_{N}'(t)&=-\frac{\beta^2}{2}\e_\theta\Pi_{N}^{t,\theta}-\frac{\beta^2}{2}\e_\theta E_{N}^{t,\theta}+O(\varepsilon)
\end{split}
\end{align}
for 
\begin{align*}
\Pi_{N}^{t,\theta}&:=\sum_{\ell,\ell'=1,2}\Bigl(\theta(\rho_{2}^{\ell ,\ell '})-\e\bigl\la \theta(\rho_{\eta(1)\wedge\eta(2) }^{\ell  ,\ell  '})\bigr\ra_{N}^{t,\theta}\Bigr),\\
E_{N}^{t,\theta}&:=\sum_{\ell  ,\ell  '=1,2}\e\bigl\la \xi (Q_{1,2}^{\ell ,\ell '})-Q_{1,2}^{\ell ,\ell '}\xi'(\rho_{\eta(1)\wedge\eta(2)  }^{\ell  ,\ell  '})+\theta(\rho_{\eta(1)\wedge\eta(2)  }^{\ell  ,\ell  '})\bigr\ra_{N}^{t,\theta}
\end{align*}
Here $O(\varepsilon)$ means that $|O(\varepsilon)|\leq C\varepsilon$ for some constant $C>0$ depending only on $\xi.$
To handle $\Pi_{N}^{t,\theta}$ and $E_{N}^{t,\theta}$, we recall a crucial consequence of the Ruelle probability cascades that from Proposition~14.3.3. \cite{Talbook2}, for any $\theta$ and $t,$
\begin{align}\label{odd:eq2}
\e \bigl\la I\bigl(\eta(1)\wedge \eta(2)=a\bigr)\bigr\ra_{N}^{t,\theta}&=m_{a+1}-m_a,\,\,a=0,1,2.
\end{align}
From this,
\begin{align}
\begin{split}\notag
\Pi_{N}^{t,\theta}
%&=\sum_{\ell,\ell'=1,2}\theta(\rho_2^{\ell,\ell'})-\sum_{\ell ,\ell '=1,2}\e\bigl\la \theta(\rho_{\eta(1)\wedge \eta(2)}^{\ell ,\ell '})\bigr\ra_{N}^{t,\theta}\\
&=\sum_{\ell ,\ell '=1,2}\Bigl(\theta(\rho_{2}^{\ell ,\ell '})-\sum_{a=0}^2\theta(\rho_a^{\ell ,\ell '})(m_{a+1}-m_a)\Bigr)\\
&=\sum_{\ell ,\ell '=1,2}\sum_{a=1}^2m_p\bigl(\theta(\rho_a^{\ell ,\ell '})-\theta(\rho_{a-1}^{\ell ,\ell '})\bigr)
\end{split}\\
\begin{split}\label{odd:eq0}
&=2\Bigl(2m_1\bigl(\theta(v)-\theta(0)\bigr)+m_2\bigl(\theta(1)-\theta(v)\bigr)\Bigr).
\end{split}
\end{align}
The control of $E_{N}^{t,\theta}$ is handled in the following two steps. 

\smallskip

{\noindent \bf Step 3: Positive semi-definiteness of $Q_{1,2}$.} We proceed to argue that the overlap matrix $Q_{1,2}$ is asymptotically symmetric and positive semi-definite. Recall the index $\omega$ from \eqref{index} and the covariance of $X_N^\omega$ from \eqref{add:eq1}. Set 
\begin{align}\label{add:eq2}
C_{k,k'}^{\omega}&=\prod_{j=1}^m\bigl\la Q_{k,k'}^{\circ r}\lambda_{l_j,l_j'},\lambda_{l_j,l_j'}\bigr\ra^{n_j}.
\end{align}
For any $n\geq 2,$ let $\mathcal{F}_n$ be the collection of all measurable functions $f$ of $(Q_{k,k'})_{1\leq k\neq k'\leq n}$ with $|f|\leq 1.$ From Theorem 3.2. \cite{Pan}, it is known that $C_{k,k'}^\omega$ satisfies the so-called Ghirlanda-Guerra identities in the average sense, i.e., for any $\omega,$
\begin{align}\label{ggi}
\lim_{N\rightarrow\infty}\e_\theta \Delta_{N}^{t,\theta,\omega}=0,
\end{align}
where 
\begin{align*}
%\label{add:eq3}
\Delta_{N}^{t,\theta,\omega}:=\sup_{n\geq 2}\sup_{f\in \mathcal{F}_n}\Bigl|\e \bigl\la fC_{1,n+1}^{\omega}\bigr\ra_{N}^{t,\theta}-\frac{1}{n}\e \bigl\la f\bigr\ra_{N}^{t,\theta}\cdot \e \bigr\la C_{1,2}^{\omega}\bigr\ra_{N}^{t,\theta}-\frac{1}{n}\sum_{k=2}^n\e \bigl\la f C_{1,k}^{\omega}\bigr\ra_{N}^{t,\theta}\Bigr|.
\end{align*}
Set 
\begin{align*}
\bar\Delta_{N}^{t,\theta}=\sum_{\omega}\frac{1}{2^{|\omega|}}\Delta_{N}^{t,\theta,\omega}.
\end{align*}
From \eqref{ggi},
\begin{align*}
\lim_{N\rightarrow\infty}\int_0^1\e_\theta\bar\Delta_{N}^{t,\theta}dt=0.
\end{align*}
In other words, $\bar\Delta_{N}^{t,\theta}$ converges to zero in $L^1$ with respect to $d\p_\theta dt.$ From this,  we can pass to a subsequence $(N_{j})_{j\geq 1}$ such that $\bar\Delta_{N_{j}}^{t,\theta}$ converges to zero $d\p_\theta dt$-a.s. Let $(t,\theta)$ a realization such that this limit holds. Note that $(Q_{k,k'})_{k,k'\geq 1}$ takes value in $[-1,1]^{ \mathbb{N}\times \mathbb{N}}$, which is a compact space with respect to the metric $$
d(q,q')=\sum_{k,k'\geq 1}2^{-(k+k')}|q_{k,k'}-q_{k,k'}'|
$$
for $q=(q_{k,k'})_{k,k'\geq 1}$ and $q'=(q_{k,k'}')_{k,k'\geq 1}$. Thus, $\e\bigl\la I\bigl( (Q_{k,k'})_{k,k'\geq 1}\in \cdot\bigr)\bigr\ra_{N_j}^{t,\theta}$ for $j\geq 1$ is a tight sequence of  probability measures on the measurable space $\bigl([-1,1]^{\mathbb{N}\times\mathbb{N}},\mathscr{B}\bigl([-1,1]^{\mathbb{N}\times\mathbb{N}}\bigr)\bigr)$. As a result, we can pass to a subsequence $(N_{j_l})_{l\geq 1}$ such that the overlap matrix $(Q_{k,k'})_{k,k'\geq 1}$ under $\e \la \cdot\ra_{N_{j_l}}^{t,\theta}$ converges  in distribution. We use $(Q_{k,k'}^{t,\theta})_{k,k'\geq 1}$ to denote this limit and set $C_{k,k'}^{t,\theta,\omega}$ via \eqref{add:eq2}. Denote $Q_n^{t,\theta}=(Q_{k,k'}^{t,\theta})_{k,k'\leq n}.$ From the convergence of $\bar\Delta_{N_{j_l}}^{t,\theta}$ and this weak convergence, we obtain that for any $n\geq 1$, index $\omega$, and any bounded measurable function $f$ depending on $Q_n^{t,\theta}$, 
\begin{align*}
\e f(Q_n^{t,\theta})C_{1,n+1}^{t,\theta,\omega}=\frac{1}{n}\e f(Q_n^{t,\theta})\e C_{1,2}^{t,\theta,\omega}+\frac{1}{n}\sum_{k=2}^n\e f(Q_n^{t,\theta}) C_{1,k}^{t,\theta,\omega}.
\end{align*}
One of the important consequences of these identities is that the overlap matrix $Q_{1,2}^{t,\theta}$ satisfies

\begin{theorem}[Theorem 4 \cite{Panchenko2015}] \label{add:thm1}
	For almost surely $\p$, $Q_{1,2}^{t,\theta}$ is symmetric positive semi-definite.
\end{theorem} 

For any $\delta\geq 0,$ set $A_{\delta}$ to be the collection of all matrix $x=(x^{\ell,\ell'})_{\ell,\ell'=1,2}$ satisfying 
\begin{align*}
&x^{1,1}\geq -\delta,\\
&x^{2,2}\geq -\delta,\\
&|x^{1,2}-x^{2,1}|\leq\delta,\\
&x^{1,1}x^{2,2}\geq x^{1,2}x^{2,1}-\delta.
\end{align*}
Note that if $x$ is a symmetric positive semi-definite matrix, then $x\in A_\delta.$ 
From Theorem \ref{add:thm1}, 
\begin{align*}
\lim_{l\rightarrow\infty}\e \bigl\la I\bigl(Q_{1,2}\in A_\delta\bigr)\bigr\ra_{N_{j_l}}^{t,\theta}=1,\,\,\forall \delta>0.
\end{align*}
Observe that this limit holds independent of the choice of the subsequence $(N_{j_l})_{l\geq 1}$ of $(N_j)_{j\geq 1}$ such that $(Q_{k,k'})_{k,k'\geq 1}$ converges weakly under $\e\la \cdot\ra_{N_{j_l}}^{t,\theta}$. Therefore, from the tightness of the sequence $\e\bigl\la I\bigl( (Q_{k,k'})_{k,k'\geq 1}\in \cdot\bigr)\bigr\ra_{N_j}^{t,\theta}$ for $j\geq 1$, we conclude that for any $\delta>0,$
\begin{align*}
\lim_{j\rightarrow\infty}\e\bigl\la I\bigl(Q_{1,2}\in A_\delta\bigr)\bigr\ra_{N_{j}}^{t,\theta}=1.
\end{align*}
Since this holds a.s. under the measure $d\p_\theta dt$, the dominated convergence theorem implies that
\begin{align}\label{add:eq4}
\lim_{j\rightarrow\infty}\int_0^1\e_\theta\e\bigl\la I\bigl(Q_{1,2}\in A_\delta\bigr)\bigr\ra_{N_{j}}^{t,\theta}dt=1,\,\,\forall \delta>0.
\end{align}

\smallskip

{\noindent \bf Step 4: Control of the error term.} Using Step 3, we now control $E_{N}^{t,\theta}$. Recall $\Gamma$ from Lemma~\ref{odd:lem1}. Note that
\begin{align*}
\xi \bigl(Q_{1,2}^{\ell,\ell'}\bigr)-Q_{1,2}^{\ell,\ell'}\xi'(\rho_{\eta(1)\wedge\eta(2)}^{\ell,\ell'})+\theta(\rho_{\eta(1)\wedge\eta(2)}^{\ell,\ell'})=\frac{1}{2}\Gamma\bigl(\rho_{\eta(1)\wedge\eta(2)}^{\ell,\ell'},Q_{1,2}^{\ell,\ell'}\bigr).
\end{align*}
Write
\begin{align*}
&E_{N}^{t,\theta}=\frac{1}{2}\bigl(E_{N,1}^{t,\theta}+E_{N,2}^{t,\theta}+E_{N,3}^{t,\theta}\bigr)
\end{align*}
for
\begin{align*}
E_{N,1}^{t,\theta}&:=\e\Bigl\la I\bigl(\eta(1)\wedge\eta(2)=0\bigr)\sum_{\ell,\ell'=1,2}\Gamma\bigl(0,Q_{1,2}^{\ell,\ell'}\bigr)\Bigr\ra_{N}^{t,\theta},\\
E_{N,2}^{t,\theta}&:= \e\Bigl\la I\bigl(\eta(1)\wedge\eta(2)=1\bigr)\sum_{\ell,\ell'=1,2}\Gamma\bigl(v,Q_{1,2}^{\ell,\ell'}\bigr)\Bigr\ra_{N}^{t,\theta},\\
E_{N,3}^{t,\theta}&:= \e\Bigl\la I\bigl(\eta(1)\wedge\eta(2)=2\bigr)\sum_{\ell,\ell'=1,2}\Gamma\bigl(\rho_{2}^{\ell,\ell'},Q_{1,2}^{\ell,\ell'}\bigr)\Bigr\ra_{N}^{t,\theta}.
\end{align*}
From the asymptotic symmetry and positive semi-definiteness \eqref{add:eq4} and Fatou's lemma, for any $\delta>0,$
\begin{align*}
&\liminf_{j\rightarrow\infty}\int_0^1 \e_\theta E_{N_j,1}^{t,\theta}dt\\
&=\liminf_{j\rightarrow\infty}\int_0^1\e_\theta\e\Bigl\la I\bigl(\eta(1)\wedge\eta(2)=0\bigr)\sum_{\ell,\ell'=1,2}\Gamma\bigl(0,Q_{1,2}^{\ell,\ell'}\bigr);Q_{1,2}\in A_\delta\bigr)\Bigr\ra_{N}^{t,\theta}dt\\
&\geq \int_0^1\liminf_{j\rightarrow\infty}\e_\theta\e\Bigl\la I\bigl(\eta(1)\wedge\eta(2)=0\bigr)\sum_{\ell,\ell'=1,2}\Gamma\bigl(0,Q_{1,2}^{\ell,\ell'}\bigr);Q_{1,2}\in A_\delta\Bigr\ra_{N}^{t,\theta}dt\\
&\geq\int_0^1\liminf_{j\rightarrow\infty}\Bigl(C_\delta\e_\theta\e\bigl\la I\bigl(\eta(1)\wedge\eta(2)=0\bigr)\bigr\ra_{N}^{t,\theta}\Bigr)dt,
\end{align*}
where $$
C_\delta:=\inf\Bigl\{\sum_{\ell,\ell'=1,2}\Gamma(0,x^{\ell,\ell'}):(x^{\ell,\ell'})_{\ell,\ell'=1,2}\in A_\delta\Bigr\}.
$$
Since the distance between $A_\delta$ and $A_0$ converges to zero as $\delta\downarrow 0$ and $\Gamma(0,x)$ is continuous on $[-1,1]$, these imply that $
\lim_{\delta \downarrow 0}C_\delta=C_0\geq 0$ by Lemma \ref{odd:lem1}. Thus, 
$$\liminf_{j\rightarrow\infty}\int_0^1 \e_\theta E_{N_j,1}^{t,\theta}dt\geq 0.$$
The same reasoning also gives the same inequality for $E_{N_j,2}^{t,\theta}.$
As for $E_{N_j,3}^{t,\theta}$, note that since $\rho_{2}^{\ell,\ell'}$ are not all the same, we can not use Lemma \ref{odd:lem1} to control $E_{N_j,3}^{t,\theta}.$ Nevertheless, from \eqref{odd:eq2}, this term is bounded above by
\begin{align*}
|E_{N_j,3}^{t,\theta}|&\leq 4p \e\bigl\la I\bigl(\eta(1)\wedge\eta(2)=0\bigr)\bigr\ra_{N_j}^{t,\theta}= 4p(m_3-m_2)=4p(1-m_2).
\end{align*}
In conclusion, we arrive at
\begin{align}\label{error}
\liminf_{j\rightarrow\infty}\int_0^1\e_\theta E_{N_j}^{t,\theta}dt\geq -4p(1-m_2).
\end{align}

\smallskip

\begin{proof}[\bf Proof of Proposition \ref{odd:lem2}] 	Note that adding a Gaussian perturbation $X_{N}^{\theta}$ increases the coupled free energy $\e CF_{N,v,\varepsilon}(\beta)$ (see, e.g., \cite[Lemma 12.2.1.]{Talbook2}), so
	\begin{align}\label{add:eq5}
	\e CF_{N,v,\varepsilon}(\beta)&\leq \psi_{N}(1).
	\end{align}
	In addition, since the variance of $X_{N}^{\theta}$ is bounded above by a constant $D>0$ independent of $N,$ it follows by releasing the overlap constraint and applying Jensen's inequality that
	\begin{align*}
	%\begin{split}\notag
	\psi_{N}(0)
	&\leq \frac{1}{N}\e\log \sum_{(\sigma^1,\sigma^2,\eta)\in \Sigma_N^2\times \mathbb{N}^2}\frac{c_\eta}{2^{2N}}\exp \beta H_{N}^0\bigl(\sigma^1,\sigma^2,\eta\bigr)+\frac{\beta^2D}{2N^{1-\chi}}\\
	&=\frac{1}{N}\e\log \sum_{\eta\in \mathbb{N}^2}c_\eta\cosh^N\beta\bigl(z_{1,j_1}^1+z_{2,j_1,j_2}^1\bigr)\cosh^N\beta\bigl(z_{1,j_1}^2+z_{2,j_1,j_2}^2\bigr)+\frac{\beta^2D}{2N^{1-\chi}}\\
	%\end{split}\\
	%\begin{split}\label{add:eq6}
	&=\frac{1}{m_1}\log \e_1\exp\Bigl( \frac{2m_1}{m_2}\log \e_2\exp \Bigl(m_2\log\cosh \beta\bigl(z^1+z^2\bigr)\Bigr)\Bigr)+\frac{\beta^2D}{2N^{1-\chi}},
	%\end{split}
	\end{align*}
	where $\e_1$ and $\e_2$ are the expectations with respect to independent
	\begin{align*}
	z^1\thicksim N\bigl(0,\xi'(v)-\xi'(0)\bigr),\\
	z^2\thicksim N\bigl(0,\xi'(1)-\xi'(v)\bigr),
	\end{align*}
	respectively. Here the second equality of the above inequality used an identity for the coupled free energy, weighted by the Ruelle probability cascades, from \cite[Proposition 14.2.2]{Talbook2}. From  \eqref{derivative} and the above inequality,
	\begin{align*}
	\psi_{N}(1)&=\psi_{N}(0)+\int_0^1\psi_{N}(t)dt\\
	&\leq \frac{1}{m_1}\log \e_1\exp \Bigl(\frac{2m_1}{m_2}\log \e_2\exp\Bigl( m_2\log\cosh \beta\bigl(z^1+z^2\bigr)\Bigr)\Bigr)\\
	&-\frac{\beta^2}{2}\int_0^1\e_\theta\Pi_{N}^{t,\theta}dt-\frac{\beta^2}{2}\int_0^1\e_\theta E_{N}^{t,\theta}dt+\frac{\beta^2D}{2N^{1-\chi}}+O(\varepsilon).
	\end{align*}
	Let $(N_j')_{j\geq 1}$ be any sequence with $N_j'\uparrow \infty$ as $j\rightarrow \infty$ such that $$
	\lim_{j\rightarrow\infty}\e CF_{N_{j}',v,\varepsilon}(\beta)=\limsup_{N\rightarrow\infty}\e CF_{N,v,\varepsilon}(\beta).
	$$
	In view of the argument of Steps 3 and 4, we can pass to a subsequence $(N_{j_l}')_{l\geq 1}$ of $(N_j')_{j\geq 1}$ such that the conclusion \eqref{error} holds for this subsequence, i.e.,
	\begin{align*}
	\liminf_{l\rightarrow\infty}\int_0^1\e_\theta E_{N_{j_l}'}^{t,\theta}dt\geq -4p(1-m_2).
	\end{align*}
	As a result, from this, \eqref{odd:eq0}, and \eqref{add:eq5},
	\begin{align*}
	\limsup_{N\rightarrow\infty}\e CF_{N,v,\varepsilon}(\beta)&=\lim_{l\rightarrow\infty}\e CF_{N_{j_l}',v,\varepsilon}(\beta)\\
	&\leq \frac{1}{m_1}\log \e_1\exp\Bigl( \frac{2m_1}{m_2}\log \e_2\exp \Bigl(m_2\log\cosh \beta\bigl(z^1+z^2\bigr)\Bigr)\Bigr)\\
	&\quad-\frac{\beta^2}{2}\Bigl(2m_1\bigl(\theta(v)-\theta(0)\bigr)+m_2\bigl(\theta(1)-\theta(v)\bigr)\Bigr)+O(\varepsilon)+2\beta^2p(1-m_2).
	\end{align*}
	Since this inequality holds for any $\varepsilon>0$ and $0<m_2< 1,$ sending $\varepsilon \downarrow 0$ and $m_2\uparrow 1$ leads to
	\begin{align*}
	\lim_{\varepsilon\downarrow 0}\limsup_{N\rightarrow\infty}\e CF_{N,v,\varepsilon}(\beta)&\leq \frac{1}{m_1}\log \e_1\exp\Bigl( 2m_1\log \e_2\cosh \beta\bigl(z^1+z^2\bigr)\Bigr)\\
	&\quad-\beta^2\Bigl(2m_1\bigl(\theta(v)-\theta(0)\bigr)+\bigl(\theta(1)-\theta(v)\bigr)\Bigr).
	\end{align*}
	By directly solving the PDE \eqref{pde} via Cole-Hopf transformation for $\gamma\in \mathcal{N}_v$ satisfying $\gamma\equiv 2m_1$, the first term of the right-hand side of the above inequality can  be written as $2\Psi_\gamma(0,0)$ and the two terms together gives $\mathcal{L}_v(\gamma)$. This completes our proof.
\end{proof}

\begin{remark}\label{rmk4}
	\rm As we have seen from the above proof, the most difficult part is to show that the error term $\int_0^1\e_\theta E_{N_j}^{t,\theta}dt$ has a nonnegative sign. In the case that $p$ is even, one does not need to add perturbation in order to achieve this. In fact, since $\xi$ is convex on $[-1,1]$ for even $p$, it is clear that $\xi(x)-y\xi'(x)+\theta(y)\geq 0$ for all $x,y\in[-1,1]$. This immediately deduces  $E_{N}^{t,\theta}\geq 0$ and then the validity of the Guerra-Talagrand inequality.
\end{remark}

\appendix
\setcounter{secnumdepth}{0}
\section{Appendix}

This appendix is devoted to establishing the existence of the limits for the auxiliary free energy $AF_N(\beta)$ and the interpolating free energy $L_N(\beta)$ defined in \eqref{af} and \eqref{af2}, based on the approach in \cite{C11}, where the author obtained a variational representation for the limiting free energy of the mixed even $p$-spin model with Curie-Weiss interaction. 
Define the free energy of the pure $p$-spin model with external field by
\begin{align*}
F_N(\beta,x)&=\frac{1}{N}\log \sum_{\sigma\in \Sigma_N}\frac{1}{2^N}\exp\Bigl(\frac{\beta}{\sqrt{2}}H_N(\sigma)+xNm_N(\sigma)\Bigr),
\end{align*} 
where $$m_N(\sigma)=\frac{1}{N}\sum_{i=1}^Nh_i\sigma_i$$ is the magnetization of $\sigma$.
It is known that the above free energy converges to a nonrandom constant, which can be computed by the Parisi formula, see \cite{Pan}. Denote this limit by $F(\beta,x)$. We show that the limiting free energies of $AF_N$ and $L_N$ can be expressed in terms of $F(\beta,x).$

     \begin{proposition}\label{add:prop1}
	Let $p\geq 2$. For any $\beta>0$ and $x>0,$ the following two equations hold a.s.,
	\begin{align}
	\begin{split}\label{add:eq-6}
	AF(\beta)&:=\lim_{N\rightarrow\infty}AF_N(\beta)=\max_{m\in [0,1]}\Bigl(F\Bigl(\beta,\frac{\beta^2pm^{p-1}}{2}\Bigr)-\frac{\beta^2(p-1)m^p}{2}\Bigr),
	\end{split}
	\\
	\begin{split}\label{add:eq-11}
	L(x)&:=\lim_{N\rightarrow\infty}L_N(\beta)=\max_{m\in [0,1]}\Bigl(F\Bigl(\beta,\frac{\beta xpm^{p-1}}{2}\Bigr)-\frac{\beta x(p-1)m^p}{2}\Bigr).
	\end{split}
	\end{align}
\end{proposition}

Before we turn to the proof of Proposition \ref{prop1}, we recall a few properties of $F_N(\beta,x)$ and $F(\beta,x).$ First, one of the important consequences of the Parisi formula states that $F(\beta,x)$ is differentiable in $x\in \mathbb{R}$ and  $\partial_xF(\beta,x)>0$ for $x>0.$ See, for instance, \cite{PanMultipSP}. For any measurable set $A\subset[-1,1]$, set 
\begin{align*}
F_N(\beta,x,A)=\frac{1}{N}\log \sum_{m_N(\sigma)\in A}\frac{1}{2^N}\exp\Bigl(\frac{\beta}{\sqrt{2}}H_N(\sigma)+xNm_N(\sigma)\Bigr).
\end{align*}
From the above differentiability of $F(\beta,x)$ in $x$, it follows from \cite[Theorem 1]{ACshort} that for any $x>0,$ the magnetization $m_N(\sigma)$ is concentrated around $\partial_xF(\beta,x)>0$ under the free energy, i.e., for any $\varepsilon>0,$
\begin{align*}
\lim_{N\rightarrow\infty}F_N\bigl(\beta,x,\bigl(\partial_xF(\beta,x)-\varepsilon,\partial_xF(\beta,x)+\varepsilon\bigr)\bigr)=F(\beta,x).
\end{align*}
This evidently implies that
\begin{align}\label{add:eq-3}
\lim_{N\rightarrow\infty}F_N\bigl(\beta,x,[0,1]\bigr)=F(\beta,x).
\end{align}

\begin{proof}[\bf Proof of Proposition \ref{add:prop1}]
	For any measurable $A\subset[-1,1]$, define
	\begin{align*}
	AF_N(\beta,A)&=\frac{1}{N}\log \sum_{m_N(\sigma)\in A}\frac{1}{2^N}\exp\Bigl(\frac{\beta }{\sqrt{2}}H_N(\sigma)+\frac{\beta^2N}{2}m_N(\sigma)^p\Bigr).
	\end{align*}
    To prove \eqref{add:eq-11}, we first consider the case that $p$ is odd. 
	Since $s^p$ is convex on $[0,1]$, 
	\begin{align}\label{ineq}
	m_N(\sigma)^p\geq m^p +p m^{p-1}\bigl(m_N(\sigma)-m\bigr)=pm^{p-1}m_N(\sigma)-(p-1)m^p,
	\end{align}
	whenever $m,m_N(\sigma)\in[0,1]$. This inequality deduces that
	\begin{align*}
	&AF_N\bigl(\beta,[0,1]\bigr)\geq F_N\Bigl(\beta,\frac{\beta^2pm^{p-1}}{2},[0,1]\Bigr)-\frac{\beta^2(p-1)m^p}{2}
	\end{align*}
	for any $m\in [0,1]$, from which letting $N\rightarrow\infty$ and using  \eqref{add:eq-3} give that 
	\begin{align}\label{add:eq-7}
	&\limsup_{N\rightarrow\infty}AF_N\bigl(\beta,[0,1]\bigr)\geq  \max_{m\in [0,1]}\Bigl(F\Bigl(\beta,\frac{\beta^2pm^{p-1}}{2}\Bigr)-\frac{\beta^2(p-1)m^p}{2}\Bigr).
	\end{align}
	To obtain the upper bound, set $m_k=k/N$ for $0\leq k\leq N.$ Write
	\begin{align*}
	&AF_N\bigl(\beta,[0,1]\bigr)\\
	&=\frac{1}{N}\log\sum_{k=0}^N\sum_{m_N(\sigma)=m_k}\frac{1}{2^N}\exp \Bigl(\frac{\beta}{\sqrt{2}}H_N(\sigma)+\frac{\beta^2pm_k^{p-1}}{2} Nm_N(\sigma)-\frac{\beta^2(p-1)m_k^p}{2}\Bigr)\\
	&\leq\frac{1}{N}\log\sum_{k=0}^N\sum_{\sigma}\frac{1}{2^N}\exp \Bigl(\frac{\beta}{\sqrt{2}}H_N(\sigma)+\frac{\beta^2pm_k^{p-1}}{2} Nm_N(\sigma)-\frac{\beta^2(p-1)m_k^p}{2}\Bigr)\\
	&=\frac{1}{N}\log\sum_{k=0}^N\exp N\Bigl(F_N\Bigl(\beta,\frac{\beta^2pm_k^{p-1}}{2}\Bigr)-\frac{\beta^2(p-1)m_k^p}{2}\Bigr)\\
	&\leq \frac{\log (N+1)}{N}+\max_{m\in[0,1]}\Bigl(F_N\Bigl(\beta,\frac{\beta^2pm^{p-1}}{2}\Bigr)-\frac{\beta^2(p-1)m^p}{2}\Bigr).
	\end{align*}
	Note that $F_N(\beta,\cdot)$ is convex and converges to $F(\beta,\cdot)$ pointwise. It follows by the classical theory of convex functions \cite{convex} that this convergence must be uniform on compact intervals. As a result,
	\begin{align*}
	\liminf_{N\rightarrow\infty}AF_N\bigl(\beta,[0,1]\bigr)&\leq \max_{m\in [0,1]}\Bigl(F\Bigl(\beta,\frac{\beta^2pm^{p-1}}{2}\Bigr)-\frac{\beta^2(p-1)m^p}{2}\Bigr).
	\end{align*}
	This and \eqref{add:eq-7} together lead to
	\begin{align}\label{add:eq-4}
	\lim_{N\rightarrow\infty}AF_N\bigl(\beta,[0,1]\bigr)&=\max_{m\in [0,1]}\Bigl(F\Bigl(\beta,\frac{\beta^2pm^{p-1}}{2}\Bigr)-\frac{\beta^2(p-1)m^p}{2}\Bigr)\Bigr).
	\end{align}
	Next since $s^p$ is concave on $[-1,0)$, the inequality \eqref{ineq} is flipped,
	\begin{align*}
	m_N(\sigma)^p\leq pm^{p-1}m_N(\sigma)-(p-1)m^p
	\end{align*}
	whenever $m,m_N(\sigma)\in [-1,0)$. Thus, for any $m\in [-1,0),$
	\begin{align*}
	&AF_N\bigl(\beta,[-1,0)\bigr)\leq F_N\Bigl(\beta,\frac{\beta^2pm^{p-1}}{2}\Bigr)-\frac{\beta^2(p-1)m^p}{2}
	\end{align*}
	and then,
	\begin{align}\label{add:eq-5}
	&\limsup_{N\rightarrow\infty}AF_N\bigl(\beta,[-1,0)\bigr)\leq \min_{m\in[-1,0]}\Bigl( F\Bigl(\beta,\frac{\beta^2pm^{p-1}}{2}\Bigr)-\frac{\beta^2(p-1)m^p}{2}\Bigr).
	\end{align}
	Denote by
	\begin{align*}
	\Gamma_+&=\max_{m\in [0,1]}\Bigl(F\Bigl(\beta,\frac{\beta^2pm^{p-1}}{2}\Bigr)-\frac{\beta^2(p-1)m^p}{2}\Bigr),\\
	\Gamma_-&=\min_{m\in[-1,0]}\Bigl( F\Bigl(\beta,\frac{\beta^2pm^{p-1}}{2}\Bigr)-\frac{\beta^2(p-1)m^p}{2}\Bigr).
	\end{align*}
	Evidently $\Gamma_+\geq \Gamma_-.$
	From \eqref{add:eq-4} and \eqref{add:eq-5}, with probability one, for any $\varepsilon>0,$ as long as $N$ is large enough,
	\begin{align*}
	e^{N(\Gamma_+-\varepsilon)}&\leq e^{N AF_N(\beta,[0,1])}\leq e^{N(\Gamma_++\varepsilon)}
	\end{align*}
	and
     \begin{align*}
	 e^{N AF_N(\beta,[-1,0))}\leq e^{N(\Gamma_-+\varepsilon)}.
	\end{align*}
	Since $e^{NAF_N(\beta)}=e^{N AF_N(\beta,[-1,0))}+e^{N AF_N(\beta,[0,1])}$, we obtain that
	\begin{align*}
	e^{N(\Gamma_+-\varepsilon)}\leq e^{NAF_N(\beta)}\leq e^{N(\Gamma_++\varepsilon)}+e^{N(\Gamma_-+\varepsilon)}.
	\end{align*}
	Consequently, taking $N^{-1}\log$ and sending $N\rightarrow\infty$ and $\varepsilon\downarrow 0,$
	\begin{align*}
	\Gamma_+=AF(\beta)&=\max\bigl(\Gamma_+,\Gamma_-\bigr)=\Gamma_+,
	\end{align*}
	where the third equality used $\Gamma_-\leq \Gamma_+.$ This establishes \eqref{add:eq-6} for odd $p.$ As for the case of even $p$, since $s^p$ is convex on $[-1,1]$, this case is identical to that of \eqref{add:eq-4}. This gives \eqref{add:eq-6}. The proof for \eqref{add:eq-11} is the same as that for \eqref{add:eq-6} without essential changes. We omit the details here. 
\end{proof}

	\bibliography{ref}
	\bibliographystyle{plain}

\end{document}